\documentclass[11pt,a4paper]{article}

\usepackage{amsmath}
\usepackage{amssymb,latexsym}
\usepackage{amsthm}
\usepackage[shortlabels]{enumitem}
\usepackage{array}
\usepackage[hang,flushmargin,bottom]{footmisc}
\usepackage[margin=0.35in,format=hang,font=small,labelfont=bf]{caption}
\usepackage[normalem]{ulem}
\usepackage{needspace}
\usepackage{graphicx,tikz}
\usepackage[colorlinks=true,urlcolor=blue,linkcolor=blue,citecolor=blue]{hyperref}
\usepackage{palatino}
\usepackage{eulervm}

\newtheorem{thm}{Theorem}[section]
\newtheorem*{thm*}{Theorem}

\newtheorem{cor}[thm]{Corollary}
\newtheorem{lemma}[thm]{Lemma}

\newtheorem{obs}[thm]{Observation}

\newtheorem*{conj*}{Conjecture}

\theoremstyle{definition}

\newtheorem{defn*}{Definition}

\newtheorem*{example*}{Example}

\newtheorem*{comment*}{Comment}

\newcommand{\CCC}{\mathcal{C}}

\newcommand{\GGG}{\mathcal{G}}

\newcommand{\PPP}{\mathcal{P}}

\newcommand{\Free}{\operatorname{Free}}
\newcommand{\Prime}{\operatorname{Prime}}
\newcommand{\Label}{\operatorname{label}}
\newcommand{\Colour}{\operatorname{colour}}
\newcommand{\cwd}{\operatorname{cwd}}
\newcommand{\rwd}{\operatorname{rwd}}
\newcommand{\tree}{\operatorname{tree}}

\usetikzlibrary{calc, backgrounds}
\usetikzlibrary{decorations.markings}
\newcommand{\mnnodearray}[2]{ 
\foreach \i in {1,...,#1}
	\foreach \j in {1,...,#2}
		\node at (\i,\j) {};
}
\newcommand{\uphoriz}[2]{ 
\pgfmathtruncatemacro\rightcol{#1 + 1}
\foreach \i in {1,...,#2}
	\foreach \j in {\i,...,#2}
		\draw (#1,\i) -- (\rightcol,\j);
}
\newcommand{\downhoriz}[2]{ 
\pgfmathtruncatemacro\rightcol{#1 + 1}
\foreach \i in {1,...,#2}
	\foreach \j in {\i,...,#2}
		\draw (#1,\j) -- (\rightcol,\i);
}
\newcommand{\upnohoriz}[2]{ 
\pgfmathtruncatemacro\rightcol{#1 + 1}
\pgfmathtruncatemacro\rowminus{#2 - 1}
\foreach \i in {1,...,\rowminus}
	{
	\pgfmathtruncatemacro\ip{\i+1}
	\foreach \j in {\ip,...,#2}
		\draw (#1,\i) -- (\rightcol,\j);
	}
}
\newcommand{\downnohoriz}[2]{ 
\pgfmathtruncatemacro\rightcol{#1 + 1}
\pgfmathtruncatemacro\rowminus{#2 - 1}
\foreach \i in {1,...,\rowminus}
	{
	\pgfmathtruncatemacro\ip{\i+1}
	\foreach \j in {\ip,...,#2}
		\draw (#1,\j) -- (\rightcol,\i);
	}
}
\newcommand{\horiz}[2]{ 
\pgfmathtruncatemacro\rightcol{#1 + 1}
\foreach \i in {1,...,#2}
		\draw (#1,\i) -- (\rightcol,\i);
}

\newcommand{\convexpath}[2]{
  [   
  create hullcoords/.code={
    \global\edef\namelist{#1}
    \foreach [count=\counter] \nodename in \namelist {
      \global\edef\numberofnodes{\counter}
      \coordinate (hullcoord\counter) at (\nodename);
    }
    \coordinate (hullcoord0) at (hullcoord\numberofnodes);
    \pgfmathtruncatemacro\lastnumber{\numberofnodes+1}
    \coordinate (hullcoord\lastnumber) at (hullcoord1);
  },
  create hullcoords
  ]
  ($(hullcoord1)!#2!-90:(hullcoord0)$)
  \foreach [
  evaluate=\currentnode as \previousnode using \currentnode-1,
  evaluate=\currentnode as \nextnode using \currentnode+1
  ] \currentnode in {1,...,\numberofnodes} {
    let \p1 = ($(hullcoord\currentnode) - (hullcoord\previousnode)$),
    \n1 = {atan2(\y1,\x1) + 90},
    \p2 = ($(hullcoord\nextnode) - (hullcoord\currentnode)$),
    \n2 = {atan2(\y2,\x2) + 90},
    \n{delta} = {Mod(\n2-\n1,360) - 360}
    in 
    {arc [start angle=\n1, delta angle=\n{delta}, radius=#2]}
    -- ($(hullcoord\nextnode)!#2!-90:(hullcoord\currentnode)$) 
  }
}

\tikzstyle{vertex}=[circle, draw, fill=black,
                        inner sep=0pt, minimum width=4pt]
\tikzstyle{every node}=[circle, draw, fill=black,
                        inner sep=0pt, minimum width=4pt]

\title{Uncountably many minimal hereditary classes of graphs of unbounded clique-width}

\author{R. Brignall \qquad D. Cocks\\
\small School of Mathematics and Statistics\\
\small The Open University, UK}

\begin{document}
\maketitle

\begin{abstract}
Given an infinite word over the alphabet $\{0,1,2,3\}$, we define a class of bipartite hereditary graphs $\mathcal{G}^\alpha$, and show that $\mathcal{G}^\alpha$ has unbounded clique-width unless $\alpha$ contains at most finitely many non-zero letters. 

We also show that $\mathcal{G}^\alpha$ is minimal of unbounded clique-width if and only if $\alpha$ belongs to a precisely defined collection of words $\Gamma$. The set $\Gamma$ includes all almost periodic words containing at least one non-zero letter, which both enables us to exhibit uncountably many pairwise distinct minimal classes of unbounded clique width, and also proves one direction of a conjecture due to Collins, Foniok, Korpelainen,  Lozin and Zamaraev. Finally, we show that the other direction of the conjecture is false, since $\Gamma$ also contains words that are \emph{not} almost periodic.
\end{abstract}

%
%
%
%
%
%
\section{Introduction}
Typically, when some graph parameter is bounded for a given graph or collection of graphs, then there exist efficient (polynomial time) algorithms for a range of problems that are in general intractable. To give two examples, Courcelle's Theorem~\cite{courcelle:the-monadic:} states that any graph property expressible in MSO$_2$ logic can be decided in linear time on graphs with bounded treewidth, and Courcelle, Makowsky and Rotics~\cite{courcelle:linear-time-sol:} showed that any graph property expressible in MSO$_1$ logic has a linear time algorithm on graphs with bounded clique-width.

While treewidth is the parameter of choice for minor-closed classes of graphs, clique-width has the property that it can remain bounded for graphs with a high edge density, and is thus of more use with \emph{hereditary} classes of graphs. Indeed, if $H$ is an induced subgraph of $G$, then the clique-width of $H$ is at most the clique-width of $G$. (For formal definitions, see Section~\ref{prelim}.) 

In light of the algorithmic consequences, a natural goal is to characterise which classes of graphs are unbounded with respect to a given parameter. 
A standard approach is to identify the \emph{minimal} classes: for example, planar graphs are the unique minimal minor-closed class of graphs of unbounded treewidth (see Robertson and Seymour~\cite{robertson:graph-minors-v:}), and circle graphs are the unique minimal vertex-minor-closed class of unbounded rank-width (or, equivalently, clique-width) -- see Geelen, Kwon, McCarty and Wollan~\cite{geelen:the-grid-theorem:}.

The situation for clique-width and hereditary classes is much more complicated, yet remains of significant interest (note that every vertex-minor-closed class is a hereditary class, but not vice-versa). First, there exist hereditary classes of graphs that have unbounded clique-width, but which contain \emph{no} minimal class of unbounded clique-width: it is well-known that the class of square grid graphs has this property; a more recent example is due to Korpelainen~\cite{korpelainen:a-new-graph:}, who also suggests possible ways to handle such classes. 

However, there do also exist minimal hereditary classes of unbounded clique-width. We refer the reader to the excellent survey by Dabrowksi, Johnson and Paulusma~\cite{dabrowski:cliquewidth} for further details of the progress in recent years. Of particular relevance here is the work of Collins, Foniok, Korpelainen and Lozin~\cite{collins:infinitely-many:}, in which a countably infinite family of minimal hereditary classes of unbounded clique-width is given. 

More precisely, the authors of~\cite{collins:infinitely-many:} construct hereditary bipartite graph classes by taking the finite induced subgraphs of an infinite graph whose vertices form a two-dimensional array and whose edges are defined by an infinite word over the alphabet $\{0,1,2\}$. They show that classes defined by an infinite periodic word over the alphabet $\{0,1\}$ are minimal of unbounded clique-width, and conjecture that a class defined by a word over the alphabet $\{0,1,2\}$ is minimal of unbounded clique-width if and only if the word is almost periodic, and not the all-zeros word. 

In this paper, we go further by proving the following results.

\begin{thm}\label{thm-main}
Let $\alpha$ be an infinite word over the alphabet $\{0,1,2,3\}$, and let $\GGG^\alpha$ be the corresponding hereditary graph class (as defined in Section~\ref{prelim}).
\begin{enumerate}[label=(\alph*)]
\item If $\alpha$ has an infinite number of non-zero letters, then $\GGG^\alpha$ has unbounded clique-width (Theorem~\ref {thm-0123unbound}).
\item If $\alpha$ is an almost periodic word with at least one non-zero letter then $\GGG^\alpha$ is minimal of unbounded clique-width (Theorem~\ref{thm-0123minim}).
\item The number of distinct minimal hereditary classes of graphs of unbounded clique-width is uncountably infinite (Theorem~\ref {cor-uncountable}).
\item Let $\Gamma$ denote the set of all recurrent words over $\{0,1,2,3\}$ with at least one non-zero letter for which the weight of the word between any two consecutive occurrences of any factor is bounded. Then $\GGG^\alpha$ is minimal of unbounded clique-width if and only if $\alpha\in\Gamma$ (Theorem~\ref{theorem-recur}).
\end{enumerate}
\end{thm}

Part~(b) of Theorem~\ref{thm-main} establishes that one direction of the conjecture in~\cite{collins:infinitely-many:} is true, although note that we go further by including the additional letter $3$ in the alphabet (which represents a different type of inter-column connection). 
The other direction of their conjecture is false: this follows from part~(d) of Theorem~\ref{thm-main}, which establishes precisely which of the classes, that are defined by words over $\{0,1,2,3\}$ and have unbounded clique-width, are minimal with this property. In Section~\ref{recur} we give an explicit counterexample of a word in $\Gamma$ that is not almost periodic.

The rest of this paper is organised as follows. Section~\ref{prelim} provides some background, and defines key definitions and concepts. Section~\ref{Sect:Unbounded} is devoted to providing that the classes defined over $\{0,1,2,3\}$ in general have unbounded clique-width. This is done from first principles for classes over the alphabet $\{2,3\}$, and then extended to the full four-letter alphabet using rank-width techniques. 

Section~\ref{Sect:Minimal} provides the central proof that if $\alpha$ is an infinite almost periodic word with at least one non-zero letter then $\GGG^\alpha$ is a minimal hereditary class of graphs of unbounded clique-width. To do this we modify the notion of cluster graphs, first used by Lozin~\cite{lozin:minimal-classes:}, and show how this can be used in conjunction with Menger's Theorem to provide an integrated proof of the minimality result. That there are uncountably many distinct minimal hereditary classes of graphs of unbounded clique width follows by considering Sturmian sequences.

Finally, in Section~\ref{recur}, we explore sequences that are recurrent but not almost periodic, and prove the precise characterisation between minimal and non-minimal hereditary classes of graphs of unbounded clique-width.

%
%
%
%
%
\section{Preliminaries} \label{prelim}

A graph $G$ is a pair of sets, vertices $V(G)$ and edges $E(G)\subseteq V(G)\times V(G)$. Unless otherwise stated, all graphs in this paper are simple, i.e. undirected, without loops or multiple edges. We denote $N(v)$ as the neighbourhood of a vertex $v$, that is, the set of vertices adjacent to $v$. 

A set of vertices is \emph{independent} if no two of its elements are adjacent. A graph is \emph{bipartite} if its vertices can be partitioned into two independent sets.

Given a graph $G(V,E)$, a subset $U \subseteq V$ and a vertex $v\in V\setminus U$, we say that $v$ \emph{distinguishes} $U$ if $v$ has both a neighbour and a non-neighbour in $U$. If $U$ is indistinguishable by the vertices outside $U$, we call $U$ a \emph{module}. A module $U$ is \emph{trivial} if $|U|=1$ or $U = V(G)$. A graph, every module of which is trivial, is called \emph{prime}. We denote the set of prime induced subgraphs of $G$, $\Prime(G)$.

We will use the notation $H \le_I G$ to denote graph $H$ is an \emph{induced subgraph} of graph $G$, meaning $H$ can be obtained from $G$ by a sequence of vertex removals. If graph $G$ does not contain the induced subgraph $H$ we say that $G$ is \emph{$H$-free}.

A class of graphs $\CCC$ is \emph{hereditary} if it is closed under taking induced subgraphs, that is $G \in \CCC$ implies $H \in \CCC$ for every induced subgraph $H$ of $G$. It is well known that for any hereditary class $\CCC$ there exists a unique (but not necessarily finite) set of minimal forbidden graphs $\{H_1, H_2, \dots\}$ such that $\CCC= \Free(H_1, H_2, \dots)$.

If $H$ is an induced subgraph of $G$, then this can be witnessed by one or more embeddings, where an \emph{embedding} of $H$ in $G$ is an injective map $\phi:V(H) \rightarrow V(G)$ such that the subgraph of $G$ induced by the vertices $\phi(V(H))$ is isomorphic to $H$. In other words, $vw \in E(H)$ if and only if $\phi(v) \phi(w) \in E(G)$.

%
%
%
%
%
\subsection{Bipartite hereditary graph classes defined by an infinite word} \label{graphdef}

The graph classes we consider are all formed by taking the set of finite induced subgraphs of an infinite graph defined on a grid of vertices. We start by defining an infinite empty graph $\PPP$ with vertices 
\[V(\PPP) = \{v_{i,j} : i, j \in \mathbb{N}\}.\] 
In general, we think of $\PPP$ as an infinite two-dimensional array in which $v_{i,j}$ represents the vertex in the $i$-th row (counting from the left) and $j$-th column (counting from the top). Hence vertex $v_{1,1}$ is in the top left corner of the grid and the grid extends infinitely to the right and downwards. The $j$-th column of $\PPP$ is the set $C_j = \{ v_{i,j} : i \in \mathbb{N}\}$, and the $i$-th row of $\PPP$ is the set $R_i = \{ v_{i,j} : j \in \mathbb{N}\}$.

We will refer to a (finite or infinite) sequence of letters chosen from a finite alphabet as a \emph{word}.  We denote by $\alpha_j$ the $j$-th letter of the word $\alpha$ and we denote $\alpha_j^k$ to be the concatenation of $k$ copies of the letter $\alpha_j$. A \emph{factor} of $\alpha$ is a contiguous subword of $\alpha$. The \emph{length} of a word $\alpha$ is the number of letters the word contains, while the \emph{weight} of $\alpha$ is the number of non-zero letters it has, which we will denote $|\alpha|_1$.

An infinite word $\alpha$ is \emph{recurrent} if each of its factors occurs in it infinitely many times. We say that $\alpha$ is \emph{almost periodic} (sometimes called \emph{uniformly recurrent} or \emph{minimal}) if for each factor $\beta$ of $\alpha$ there exists a constant $\mathcal{L}(\beta)$ such that every factor of $\alpha$ of length at least $\mathcal{L}(\beta)$ contains $\beta$ as a factor. Finally, $\alpha$ is \emph{periodic} if there is a positive integer $p$ such that $\alpha_k=\alpha_{k+p}$ for all $k$. Clearly, every periodic word is almost periodic, and every almost periodic word is recurrent. 

Let $\alpha$ be an infinite word such that $\alpha_j \in \{0,1,2,3\}$ for each natural $j$. We define a family of infinite graphs $\{\PPP^\alpha\}$ with vertices $V(\PPP)$, and with edges between consecutive columns $C_j, C_{j+1}$ of $V(\PPP^\alpha)$, such edges determined by the letters of the word $\alpha$.
\begin{enumerate}[label=(\roman*)]
\item If $\alpha_j=0$ then the edges between $C_j$ and $C_{j+1}$ are given by  $\{(v_{i,j}, v_{i,j+1}) : i \in \mathbb{N}\}$ (i.e.\ a matching).
\item If $\alpha_j=1$ then the edges between $C_j$ and $C_{j+1}$ are given by  $\{(v_{i,j}, v_{k,j+1}) : i \neq k; i,k \in \mathbb{N}\}$ (i.e.\ the complement of a matching).
\item If $\alpha_j=2$ then the edges between $C_j$ and $C_{j+1}$ are given by  $\{(v_{i,j}, v_{k,j+1}) : i \le k; i,k \in \mathbb{N}\}$.
\item If $\alpha_j=3$ then the edges between $C_j$ and $C_{j+1}$ are given by  $\{(v_{i,j}, v_{k,j+1}) : i \ge k; i,k \in \mathbb{N}\}$.
\end{enumerate}
This notation matches and extends that used in~\cite{collins:infinitely-many:}.

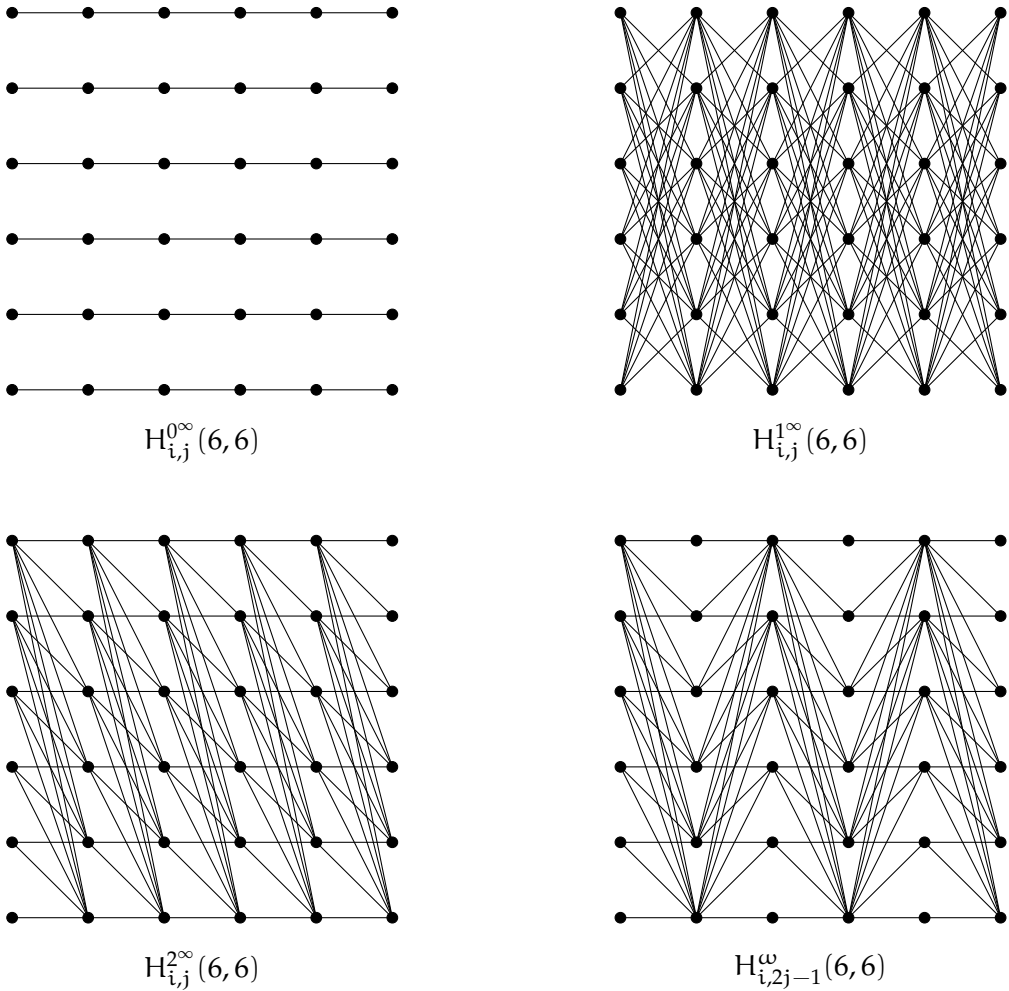
\begin{figure}[ht]
{\centering
\begin{tikzpicture}[scale=1]
	\mnnodearray{6}{6}
		\foreach \col in {1,2,3,4,5}
			\horiz{\col}{6};
		\node[draw=none,fill=none] at (3.5,0.3) {$H^{0^\infty}_{i,j}(6,6)$};
	\begin{scope}[shift={(8,0)}]
	\mnnodearray{6}{6}
		\foreach \col in {1,2,3,4,5}
			\downnohoriz{\col}{6};
		\foreach \col in {1,2,3,4,5}
			\upnohoriz{\col}{6};
		\node[draw=none,fill=none] at (3.5,0.3) {$H^{1^\infty}_{i,j}(6,6)$};
	\end{scope}
	\begin{scope}[shift={(0,-7)}]
	\mnnodearray{6}{6}
		\foreach \col in {1,2,3,4,5}
			\downhoriz{\col}{6};
		\node[draw=none,fill=none] at (3.5,0.3) {$H^{2^\infty}_{i,j}(6,6)$};
	\end{scope}		
	\begin{scope}[shift={(8,-7)}]
	\mnnodearray{6}{6}
		\foreach \col in {1,3,5}
			\downhoriz{\col}{6};
		\foreach \col in {2,4}
			\uphoriz{\col}{6};
		\node[draw=none,fill=none] at (3.5,0.3) {$H^{\omega}_{i,2j-1}(6,6)$};
	\end{scope}
\end{tikzpicture}\par }
\caption{Bipartite graphs defined on a $6 \times 6$ square grid.}
	\label{fig:squares}
\end{figure}

For ease of reference we will denote $H^\alpha_{i,j}(m,n)$ as the $m \times n$ induced subgraph of $\PPP^\alpha$ formed from the vertices 
$\{v_{x,y} : x = i, i+1, \dots , i+m-1, y = j, j+1, \dots , j+n-1\}$. See Figure~\ref{fig:squares}.

Let $\GGG^\alpha$ denote the class of all finite induced subgraphs of $\PPP^\alpha$. By definition, $\GGG^\alpha$ is a hereditary class, and any graph $G\in\GGG^\alpha$ can be witnessed by an embedding into the infinite graph $\PPP^\alpha$. Given such an embedding of $G$ into $\PPP^\alpha$, we will be especially interested in the induced subgraphs of $G$ that occur in two adjacent columns: an \emph{$\alpha_j$-link} is the induced subgraph of $G$ on the vertices of $G\cap (C_j\cup C_{j+1})$, and will be denoted by $G_j$.

Letting $2^\infty$ stand for the infinite word of all $2$s, we note that $\GGG^{2^\infty}$ is the class of \emph{bipartite permutation} graphs. Note, too, that $\GGG^{3^\infty}=\GGG^{2^\infty}$ (this can be seen by considering a vertical reflection of the 2-dimensional array), so it is not necessarily the case that two words $\alpha$ and $\beta$ give rise to distinct hereditary classes. 

On the other hand, the word with alternating 2s and 3s, which throughout we will denote by $\omega=232323\cdots$, defines the class $\GGG^\omega$ which is distinct from $\GGG^{2^\infty}$. Indeed, the graph,\par
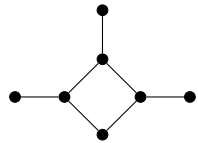
{\centering\begin{tikzpicture}[scale=0.5]
	\node (A) at (0,0) {};
	\node (B) at (-1,1) {};
	\node (C) at (1,1) {};
	\node (D) at (0,2) {};
	\node (E) at (0,3.3) {};
	\node (F) at (-2.3,1) {};
	\node (G) at (2.3,1) {};
	\draw (F) -- (B) -- (A) -- (C) -- (G)  (B) -- (D) -- (C) (D) -- (E);  
\end{tikzpicture}\par}
embeds in $\PPP^\omega$, but is one of the minimal forbidden induced subgraphs of permutation graphs.


%
%
%
%
%
\subsection{Clique-width and rank-width}\label{cliquerank}

Clique-width is a graph width parameter introduced by Courcelle, Engelfriet and Rozenberg in the 1990s~\cite{courcelle:handle-rewritin:}. A recent survey of clique-width for hereditary graph classes is~\cite{dabrowski:cliquewidth}. The clique-width of a graph is denoted $\cwd(G)$ and is defined as the the minimum number of labels needed to construct $G$ by means of the following four graph operations:

\begin{enumerate}[label=(\alph*)]
\item creation of a new vertex $v$ with label $i$ (denoted $i(v)$),
\item taking the disjoint union of two previously-constructed labelled graphs $G$ and $H$ (denoted $G \oplus H$),
\item adding an edge between every vertex labelled $i$ and every vertex labelled $j$ for distinct $i$ and $j$ (denoted $\eta_{i,j}$) and
\item giving all vertices labelled $i$ the label $j$ (denoted $\rho_{i \rightarrow j}$).
\end{enumerate}
Every graph can be defined by an algebraic expression $\tau$ using the four operations above, which we will refer to as a \emph{clique-width expression}. This expression is called a $k$-expression if it uses $k$ different labels.

Alternatively, any clique-width expression $\tau$ defining $G$ can be represented as a rooted tree, $\tree(\tau)$, whose leaves correspond to the operations of vertex creation, the internal nodes correspond to the $\oplus$-operation, and the root is associated with $G$. The operations $\eta$ and $\rho$ are assigned to the respective edges of  $\tree(\tau)$.

A related parameter is that of rank-width, which was introduced by Oum and Seymour in 2006 \cite{oum:approximating-c:}. They showed that the measures are closely related through the inequality
\[\rwd(G)\le \cwd(G)\le 2^{\rwd(G)+1}-1 \]
so that a graph class has bounded clique-width if and only if it has bounded rank-width.

For a graph $G$ and a vertex $v$, the \emph{local complementation} at $v$ is the operation that replaces the subgraph induced by the neighbourhood of $v$ with its complement. For a graph $G$  and an edge $vw$, the graph obtained by \emph{pivoting} $vw$  is the graph obtained by applying local complementation  at $v$, then at $w$ and then at $v$ again. When $G$ is bipartite, Oum showed in \cite{oum:rank-width-and:} that this is equivalent to complementing the edges between $N(v)\backslash w$ and $N(w)\backslash v$. 

We will use the notation $H \le_V G$ to denote graph $H$ is a \emph{vertex-minor} of graph $G$, meaning $H$ can be obtained from $G$ by a sequence of vertex removals and local complementations. A useful result is the following.
\begin{lemma}[Oum~\cite{oum:rank-width-and:}]\label{vminor}
If $H \le_V G$ then $\rwd(H) \le \rwd(G)$.
\end{lemma}

%
%
%
%
%
%
\section{Graph classes with unbounded clique-width} \label{Sect:Unbounded}

In this section we show that the graph classes $\{\GGG^\alpha\}$, where $\alpha$ is an infinite word over $\{0,1,2,3\}$ that contains infinitely many letters from $\{1,2,3\}$, have unbounded clique-width. We will extend the methods of Golumbic and Rotics~\cite{golumbic:on-the-clique-width:} and Brandst\"adt and Lozin~\cite{brandstadt:clique-width:}, the latter of which proved that the clique-width of the class of bipartite permutation graphs (i.e.\ $\GGG^{2^\infty}$) is unbounded. We deal with $\{2,3\}$ graphs first as these are susceptible to direct proof methods. We can then use rank-width/local complementation techniques following Collins et al \cite{collins:infinitely-many:} to show that the more complex $\{0,1,2,3\}$ graphs have a suitable $\{2,3\}$ vertex-minor to prove that these too have unbounded clique-width. 

In the case of binary words, the following result covers what we require.
\begin{lemma}[Collins, Foniok, Korpelainen, Lozin and Zamaraev~\cite{collins:infinitely-many:}] \label{lem-01unbound} 
If $\alpha$ is an infinite binary word containing infinitely many $1$s then the graph class $\GGG^\alpha$ has unbounded clique-width.
\end{lemma}

%
%
%
%
%
\subsection{\texorpdfstring{$\{2,3\}$}{\{2,3\}} graph classes with unbounded clique-width}

To assist with determining the clique-width of these graph classes, it is helpful to consider a sequence $\{W^\alpha_n\}$ of graphs which are constructed as follows.

The vertices of $W^\alpha_n$ are an $n \times n$ array with vertex $u_{i,j}$ in row i and column j. These vertices are partitioned according to the diagonal they are on, such that $u_{i,j}$ is on diagonal $D_{i+j-1}$. Hence, $D_1=\{u_{1,1}\}$, $D_2=\{u_{2,1},u_{1,2}\}$, and so on.

$W^\alpha_n$ has an edge $(u_{i,j},u_{k,l})$ whenever $u_{i,j} \in D_{m}$, $u_{k,l} \in D_{m+1}$ and either :
\begin{enumerate}[label=(\alph*)]
\item $\alpha_{m}=2$ and $k \ge i$, or
\item $\alpha_{m}=3$ and $l \ge j$.
\end{enumerate}
Observe that these edges always create an $n \times n$ square grid. Furthermore, the edges between two consecutive diagonals $D_m$ and $D_{m+1}$ of $W^\alpha_n$ correspond to the edges between two consecutive columns $C_m$ and $C_{m+1}$ of $\PPP^\alpha$.

\begin{lemma}\label{Wn}
For all $n \in \mathbb{N}$, $W^\alpha_n$ can be embedded into $\PPP^\alpha$.
\end{lemma}
\begin{proof}
For each vertex $u_{x,y}$ of $W^\alpha_n$ we will identify a vertex $v_{i,j}$ of $\PPP^\alpha$ so that we can define an injective map $\phi: V(W^\alpha_n) \rightarrow V(\PPP^\alpha)$ such that $\phi(u_{x,y})=v_{i,j}$ as follows:  

If $(x,y) = (1,1)$ then $(i,j)=(n,1)$, otherwise $(i,j)$ is given by:
\begin{align*}
&i=n+x-1-\sum_{m=1}^{x+y-2} \mathbf{1}_m \\
&j=x+y-1
\end{align*}
where 
\[
\mathbf{1}_m =
\begin{cases}
1&\text{ if } \alpha_m=3,\\
0&\text{ if } \alpha_m=2.
\end{cases}
\]

It can be seen that the subgraph of $\PPP^\alpha$ induced by the vertices $\phi(V(W^\alpha_n))$ is isomorphic to $W^\alpha_n$. 
\end{proof}
An example of embedding $W^\omega_5$ is shown in Figure \ref{fig:embed}.

\newcommand{\mnnodearrayfat}[2]{ 
\foreach \i in {1,...,#1}
	\foreach \j in {1,...,#2}
		\node at (\i,\j) [vertex2] {};
}

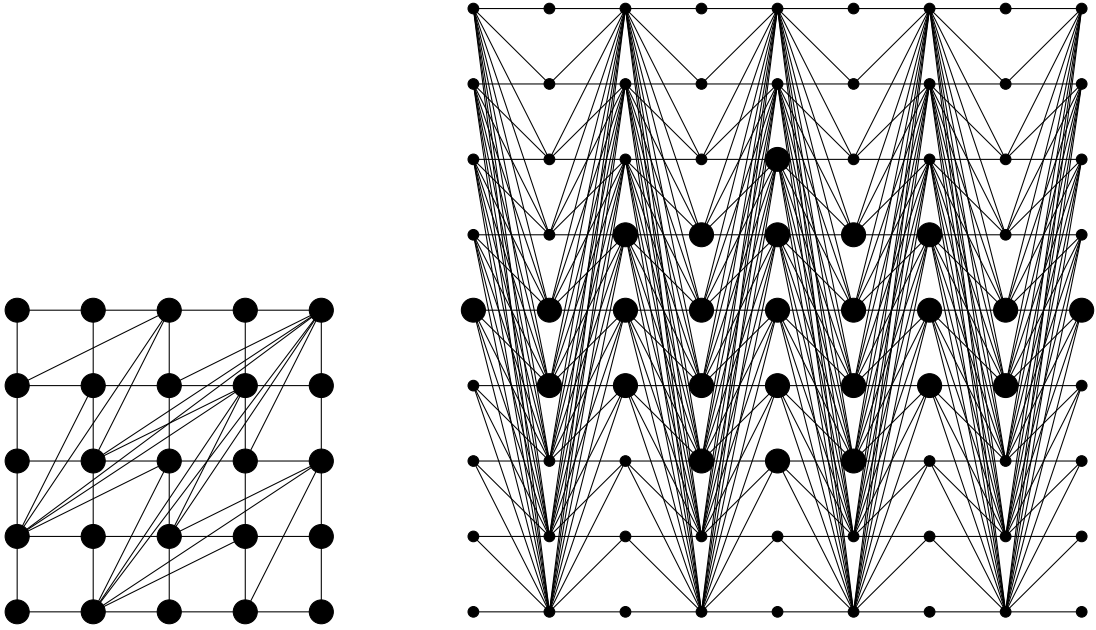
\begin{figure}\centering

\begin{tikzpicture}[scale=1,
	vertex2/.style={circle,draw,minimum size=9,fill=black},]

	\mnnodearrayfat {5}{5}
		\draw (1,5) -- (2,5);	\draw (1,5) -- (1,4);
		\draw (1,4) -- (1,3);	\draw (1,4) -- (2,4);
		\draw (1,4) -- (3,5);	\draw (2,5) -- (2,4);
		\draw (2,5) -- (3,5);	\draw (3,5) -- (4,5);
		\draw (3,5) -- (3,4);	\draw (3,5) -- (2,3);
		\draw (3,5) -- (1,2);	\draw (2,4) -- (3,4);
		\draw (2,4) -- (2,3);	\draw (2,4) -- (1,2);
		\draw (1,3) -- (2,3);	\draw (1,3) -- (1,2);
		\draw (1,2) -- (1,1);	\draw (1,2) -- (2,2);
		\draw (1,2) -- (3,3);	\draw (1,2) -- (4,4);
		\draw (1,2) -- (5,5);	\draw (2,3) -- (2,2);
		\draw (2,3) -- (3,3);	\draw (2,3) -- (4,4);
		\draw (2,3) -- (5,5);	\draw (3,4) -- (3,3);
		\draw (3,4) -- (4,4);	\draw (3,4) -- (5,5);
		\draw (4,5) -- (4,4);	\draw (4,5) -- (5,5);
		\draw (5,5) -- (5,4);	\draw (5,5) -- (4,3);
		\draw (5,5) -- (3,2);	\draw (5,5) -- (2,1);
		\draw (4,4) -- (5,4);	\draw (4,4) -- (4,3);
		\draw (4,4) -- (3,2);	\draw (4,4) -- (2,1);
		\draw (3,3) -- (4,3);	\draw (3,3) -- (3,2);
		\draw (3,3) -- (2,1);	\draw (2,2) -- (3,2);
		\draw (2,2) -- (2,1);	\draw (1,1) -- (2,1);
		\draw (2,1) -- (3,1);	\draw (2,1) -- (4,2);
		\draw (2,1) -- (5,3);	\draw (3,2) -- (3,1);
		\draw (3,2) -- (4,2);	\draw (3,2) -- (5,3);
		\draw (4,3) -- (4,2);	\draw (4,3) -- (5,3);
		\draw (5,4) -- (5,3);	\draw (5,3) -- (5,2);	
		\draw (5,3) -- (4,1);	\draw (4,2) -- (5,2);	
		\draw (4,2) -- (4,1);	\draw (3,1) -- (4,1);	
		\draw (4,1) -- (5,1);	\draw (5,2) -- (5,1);	
	
	\begin{scope}[shift={(6,0)}]
	
	\mnnodearray{9}{9}
		\foreach \col in {1,3,5,7}
			\downhoriz{\col}{9};
		\foreach \col in {2,4,6,8}
			\uphoriz{\col}{9};
	
	\node at (1,5) [vertex2]{};
	\node at (2,5) [vertex2]{};\node at (2,4) [vertex2]{};
	\node at (3,4) [vertex2]{};\node at (3,5) [vertex2]{};
	\node at (3,6) [vertex2]{};
	\node at (4,6) [vertex2]{};\node at (4,5) [vertex2]{};
	\node at (4,4) [vertex2]{};\node at (4,3) [vertex2]{};
	\node at (5,3) [vertex2]{};\node at (5,4) [vertex2]{};
	\node at (5,5) [vertex2]{};\node at (5,6) [vertex2]{};
	\node at (5,7) [vertex2]{};
	\node at (6,6) [vertex2]{};\node at (6,5) [vertex2]{};
	\node at (6,4) [vertex2]{};\node at (6,3) [vertex2]{};
	\node at (7,4) [vertex2]{};\node at (7,5) [vertex2]{};
	\node at (7,6) [vertex2]{};
	\node at (8,5) [vertex2]{};\node at (8,4) [vertex2]{};
	\node at (9,5) [vertex2]{};
			
	\end{scope}
	
\end{tikzpicture}
\caption{Embedding $W^\omega_5$ in $H^\omega_{1,1}(9,9)$ (shown by large vertices)}
	\label{fig:embed}

\end{figure}

We will now calculate a lower bound for the clique-width of $W^\alpha_n$ by demonstrating a minimum number of labels needed to construct $W^\alpha_n$ using the allowed four graph operations.

Let $\tau$ be a clique-width expression defining $W^\alpha_n$ and $\tree(\tau)$ the rooted tree representing $\tau$. The subtree of $\tree(\tau)$ rooted at a node $x$ will be denoted $\tree(x,\tau)$. This subtree corresponds to a subgraph of $W^\alpha_n$ we will call $W_x$. 

Let $a$ be the lowest $\oplus$ node in $\tree(\tau)$ such that $W_a$  contains a full row and a full column of $W^\alpha_n$. We denote by $r$ and $b$ the two children of $a$ in $\tree(\tau)$. Let us colour the vertices of $W_r$ and $W_b$ red and blue, respectively, and all the other vertices in $W^\alpha_n$ white. We let $\Colour(v)$ denote the colour of a vertex $v$ as described above, and $\Label(v)$ denote the label of vertex $v$ (if any) at node $a$. 

We assume that there is neither a blue nor a red column in $W^\alpha_n$. For if $W^\alpha_n$ contains a blue column then obviously it cannot contain a red row, and it cannot contain a blue row due to the choice of $a$.  Likewise, if $W^\alpha_n$ contains a red column. Hence, if it does include a blue or red column, as a consequence of the symmetry of $W^\alpha_n$, we can apply similar types of argument to those that follow to the rows instead of the columns to deliver the same result. 

\begin{obs}\label{obs1}
Suppose $u$, $v$, $w$ are three vertices such that $u$ and $v$ are non-white, $ uw \in  E(W^\alpha_n)$ but $ vw \notin  E(W^\alpha_n)$, and $\Label(w)\neq \Label(u)$. Then $u$ and $v$ must have different labels at node $a$ because the edge $uw$ still needs to be created, whilst respecting the non-adjacency of $v$ and $w$.
\end{obs}

Let us denote a row without white vertices by $r$. From the foregoing, we know that every column $j$ in $W^\alpha_n$ must have a vertex with different colour than that of $v_{r,j}$. We denote by $x_j$ a nearest to $v_{r,j}$ vertex in the same column with $\Colour(x_j) \neq \Colour(v_{r,j})$. We then let $y_j$ denote $v_{i+1,j}$ if $i < r$ or $v_{i-1,j}$ if $i > r$. Notice that $y_j$ is non-white and $x_j$ is adjacent to $y_j$. This creates two sets of $n$ vertices, $X = \{x_j: j = 1, \dots, n\}$ and  $Y = \{y_j: j = 1,\dots, n\}$ where $x_j$ is adjacent to $y_j$ for each $j$. 

We examine first the case for the word $\omega=2323\cdots$, the infinite word with alternating $2$s and $3$s, as this is relatively simple whilst demonstrating the technique required.

\begin{lemma}\label{Womega}
Let $\omega$ be the infinite word with alternating $2$s and $3$s. Then
\[\cwd(W^\omega_n) \ge n/2.\]
\end{lemma}
\begin{proof}
We claim that for $W^\omega_n$, no three vertices in $Y$ can have the same label. From this the lemma will follow since the $n$ vertices in $Y$ are labelled with at least $n/2$ labels at node $a$. Suppose for a contradiction that vertices $Y^\prime = \{y_{j_1}, y_{j_2}, y_{j_3}\} \subseteq Y$ have the same label at $a$. Without loss of generality we assume that $j_1 < j_2 < j_3$. Let $X^\prime = \{x_{j_1}, x_{j_2}, x_{j_3}\} \subseteq X$ be the corresponding subset of $X$. 

Assume a vertex $y_{j_i}$ in $Y^\prime$ is not adjacent to a vertex $x_{j_k}$ in $X^\prime$. Clearly, $i \neq k$ and hence vertices $y_{j_i}$, $y_{j_k}$ and $x_{j_k}$ form a triple as described in Observation \ref{obs1}, so $y_{j_i}$ and $y_{j_k}$ have different labels, a contradiction. Hence, each vertex in $Y^\prime$ must be adjacent to each vertex in $X^\prime$.

\begin{obs}\label{obs2}
\begin{enumerate}[label=(\alph*)]
\item if a vertex $v_{i,j}$ lies on an odd diagonal $D_{2m-1}$  then it has only one neighbour to its right and all its neighbours lie on the even diagonals $D_{2m}$  and $D_{2m+2}$;
\item if a vertex $v_{i,j}$ lies on an even diagonal $D_{2m}$  then it has only one neighbour to its left and all its neighbours lie on the odd diagonals $D_{2m-1}$  and $D_{2m+1}$.
\end{enumerate}
\end{obs}
 
The leftmost vertex in $X^\prime$, $x_{j_1}$, must have two neighbours in $Y^\prime$ to its right, so from Observation \ref{obs2}, $x_{j_1}$ must sit on an even diagonal and all the vertices in $Y^\prime$ must sit on odd diagonals. On the other hand, the rightmost vertex in $X^\prime$, $x_{j_3}$, must have two neighbours in $Y^\prime$ to its left, so from Observation \ref{obs2}, $x_{j_3}$ must sit on an odd diagonal and all the vertices in $Y^\prime$ must sit on even diagonals. We have a contradiction and hence no three vertices in $Y$ can have the same label, as required. 
\end{proof}

The extension to consider arbitrary words over $\{2,3\}$ is similar, but the behaviour of the diagonals given in Observation~\ref{obs2} represent just two of several possible types of behaviour.

\begin{lemma}\label{Walpha}
If $\alpha$ is any infinite word over the alphabet $\{2,3\}$ then 
\[\cwd(W^\alpha_n) \ge n/4.\] 
\end{lemma}
\begin{proof}
We claim that for $W^\alpha_n$, no five vertices in $Y$ can have the same label. From this the lemma will follow since the $n$ vertices in $Y$ are labelled with at least $n/4$ labels at node $a$.

Suppose for a contradiction that vertices $Y^\prime = \{y_{j_1}, y_{j_2}, y_{j_3}, y_{j_4}, y_{j_5}\} \subseteq Y$ have the same label at $a$. Without loss of generality we assume that $j_1 <j_2 < j_3 < j_4 <j_5$.  Let $X^\prime = \{x_{j_1}, x_{j_2}, x_{j_3}, x_{j_4}, x_{j_5}\} \subseteq X$ be the corresponding subset of $X$. As in Lemma~\ref{Womega}, we must have each vertex in $Y^\prime$ adjacent to each vertex in $X^\prime$.

If a vertex lies on diagonal $D_m$ we say it has \emph{type} $\alpha_{m-1}\alpha_m$. Hence, we have 4 possible types of vertex, namely $22$, $23$, $32$ and $33$. The distinguishing feature of each is its neighbourhood. Each has two \emph{branches} to its neighbourhood on diagonals $D_{m-1}$ and $D_{m+1}$ -- see Figure~\ref{fig:neighbour}.
\begin{enumerate}[label=(\alph*)]
\item a type $22$ vertex neighbourhood has one branch ($D_{m-1}$) extending up/right and one ($D_{m+1}$) extending down/left,
\item a type $23$ vertex neighbourhood has two branches ($D_{m-1}$ and $D_{m+1}$) extending up/right,
\item a type $32$ vertex neighbourhood has two branches ($D_{m-1}$ and $D_{m+1}$) extending down/left,
\item a type $33$ vertex neighbourhood has one branch ($D_{m-1}$) extending down/left and one ($D_{m+1}$) extending up/right.
\end{enumerate}

\begin{figure}\centering

\begin{tikzpicture}[scale=0.5,
	vertex1/.style={circle,draw,minimum size=10, fill=none},
	vertex2/.style={circle,draw,minimum size=8,fill=black},]
	
	\mnnodearray{7}{7}
		\node at (4,4) [vertex1]{};
		\node at (3,4) [vertex2]{}; \node at (4,5) [vertex2]{};
		\node at (5,6) [vertex2]{};	\node at (6,7) [vertex2]{};
		\node at (5,4) [vertex2]{};	\node at (4,3) [vertex2]{};
		\node at (3,2) [vertex2]{}; \node at (2,1) [vertex2]{};

		\node at (4,1) [label=below:{$22$}]{};

		\draw (4,4) -- (3,4);	\draw (4,4) -- (4,5);
		\draw (4,4) -- (5,6);	\draw (4,4) -- (6,7);
		\draw (4,4) -- (5,4);	\draw (4,4) -- (4,3);
		\draw (4,4) -- (3,2);	\draw (4,4) -- (2,1);

	\begin{scope}[shift={(8,0)}]
	\mnnodearray{7}{7}
		\node at (4,4) [vertex1]{};
		\node at (3,4) [vertex2]{}; \node at (4,5) [vertex2]{};
		\node at (5,6) [vertex2]{};	\node at (6,7) [vertex2]{};	
		\node at (4,3) [vertex2]{};	\node at (5,4) [vertex2]{};
		\node at (6,5) [vertex2]{}; \node at (7,6) [vertex2]{};

		\node at (4,1) [label=below:{$23$}]{};
		
		\draw (4,4) -- (3,4);	\draw (4,4) -- (4,5);
		\draw (4,4) -- (5,6);	\draw (4,4) -- (6,7);
		\draw (4,4) -- (4,3);	\draw (4,4) -- (5,4);
		\draw (4,4) -- (6,5);	\draw (4,4) -- (7,6);
			
	\end{scope}
	
	\begin{scope}[shift={(16,0)}]
	\mnnodearray{7}{7}
		\node at (4,4) [vertex1]{};
		\node at (4,5) [vertex2]{}; \node at (3,4) [vertex2]{};
		\node at (2,3) [vertex2]{};	\node at (1,2) [vertex2]{};
		\node at (5,4) [vertex2]{};	\node at (4,3) [vertex2]{};
		\node at (3,2) [vertex2]{}; \node at (2,1) [vertex2]{};

		\node at (4,1) [label=below:{$32$}]{};

		\draw (4,4) -- (4,5);	\draw (4,4) -- (3,4);
		\draw (4,4) -- (2,3);	\draw (4,4) -- (1,2);
		\draw (4,4) -- (5,4);	\draw (4,4) -- (4,3);
		\draw (4,4) -- (3,2);	\draw (4,4) -- (2,1);
			
	\end{scope}	
	
		\begin{scope}[shift={(24,0)}]
	\mnnodearray{7}{7}
		\node at (4,4) [vertex1]{};
		\node at (4,5) [vertex2]{}; \node at (3,4) [vertex2]{};
		\node at (2,3) [vertex2]{};	\node at (1,2) [vertex2]{};
		\node at (4,3) [vertex2]{};	\node at (5,4) [vertex2]{};
		\node at (6,5) [vertex2]{}; \node at (7,6) [vertex2]{};

		\node at (4,1) [label=below:{$33$}]{};

		\draw (4,4) -- (4,5);	\draw (4,4) -- (3,4);
		\draw (4,4) -- (2,3);	\draw (4,4) -- (1,2);
		\draw (4,4) -- (4,3);	\draw (4,4) -- (5,4);
		\draw (4,4) -- (6,5);	\draw (4,4) -- (7,6);
			
	\end{scope}	
	
\end{tikzpicture}\par
\caption{Neighbourhood types (large vertices are neighbours of central vertex)}
	\label{fig:neighbour}
\end{figure}
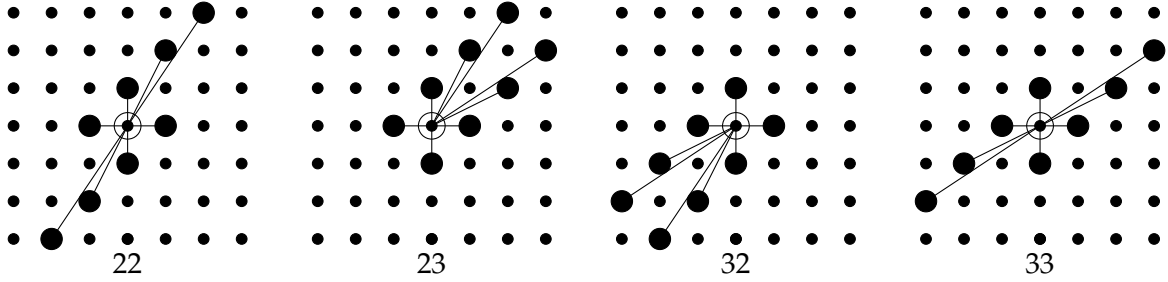

We have $5$ vertices in $X^\prime$ each of which could be one of $4$ types, so potentially there are $4^5 = 1,024$ different cases, but we can handle these in groups that depend on the type of the two outside vertices $x_{j_1}$ and $x_{j_5}$. 
 
\emph{Case 1}: $x_{j_1}$ and $x_{j_5}$ are each of type $22$ or $33$. 
 
Since $x_{j_1}$ must be adjacent to $4$ vertices in $Y^\prime$ to its right and $x_{j_5}$ must be adjacent to $4$ vertices in $Y^\prime$ to its left, then all the vertices in $Y^\prime$ must be on the same diagonal (as the neighbourhood of $x_{j_1}$ has only one branch going rightwards and the neighbourhood of $x_{j_5}$ has only one branch going leftwards). But now there is no possible neighbourhood type for $x_{j_3}$ that could make it simultaneously adjacent to $y_{j_1}$ and $y_{j_5}$, so we have a contradiction.
 
\emph{Case 2}: $x_{j_1}$ and $x_{j_5}$ are each of type $23$ or $32$. 
 
In Lemma \ref{Womega} we dealt with the case where there are $3$ vertices in $X^\prime$ that are either $23$ or $32$ types so we only need to consider the case where the middle $3$ vertices in $X^\prime$ are type $22$ or $33$. Furthermore, as $x_{j_1}$ must be adjacent to $4$ vertices in $Y^\prime$ to its right and $x_{j_5}$ must be adjacent to $4$ vertices in $Y^\prime$ to its left, then $x_{j_1}$ can only be of type $23$ and $x_{j_1}$ type $32$.
Then considering vertex $x_{j_3}$ we must have all three vertices $x_{j_1}$, $x_{j_3}$ and $x_{j_5}$ on the same diagonal otherwise $x_{j_3}$ cannot be adjacent to both $y_{j_1}$ and $y_{j_5}$. Lastly, $x_{j_2}$ and $x_{j_4}$ must be of the same type as $x_{j_3}$ (i.e. either type $22$ or $33$) and on the same diagonal as  $x_{j_3}$ in order to be adjacent to both $y_{j_1}$ and $y_{j_5}$. But then they cannot both be adjacent to $y_{j_3}$ and we have a contradiction.

\emph{Case 3}: One of $x_{j_1}$ and $x_{j_5}$ is of type $22$ or $33$ and the other is of type $23$ or $32$. 

Suppose $x_{j_1}$ is type $22$ and $x_{j_5}$ is type $32$, then $y_{j_3}$, $y_{j_4}$ and $y_{j_5}$ must be on the same diagonal. $x_{j_3}$ cannot be type $23$ or $32$ as it could not be adjacent to both $y_{j_1}$ and $y_{j_5}$. 

If $x_{j_3}$ is type $33$ then it must be on the same diagonal as $x_{j_1}$ in order for  both $x_{j_1}$ to be adjacent to $y_{j_3}$ and $x_{j_3}$ to be adjacent to $y_{j_1}$. But then $x_{j_2}$ cannot be any of the $4$ types and still be simultaneously adjacent to $y_{j_1}$, $y_{j_3}$ and $y_{j_5}$, hence we have a contradiction.

Likewise if $x_{j_3}$ is type $22$ then it must be on the diagonal $D_{m-2}$ (where $x_{j_1}$ is on $D_{m}$) in order for  both $x_{j_1}$ to be adjacent to $y_{j_3}$ and $x_{j_3}$ to be adjacent to $y_{j_1}$. But then, as before, $x_{j_2}$ cannot be any of the $4$ types and still be simultaneously adjacent to $y_{j_1}$, $y_{j_3}$ and $y_{j_5}$, hence we have a contradiction.

For the other Case 3 combinations (e.g.\ $x_{j_1}$ is type $33$ and $x_{j_5}$ is type $32$) we can use an analogous argument and each time reach a contradiction. We omit the details.

We have now considered all possible combinations of vertex type and each one leads to a contradiction. Hence, no five vertices in $Y$ can have the same label.
\end{proof}

\begin{lemma}\label{lem-23unbound}
If $\alpha$ is any infinite word over the alphabet $\{2,3\}$ then the graph class $\GGG^\alpha$ is a hereditary class of graphs of unbounded clique-width.
\end{lemma}
\begin{proof}
This follows immediately from Lemmas \ref{Womega} and \ref{Walpha}.
\end{proof}

%
%
%
%
%
\subsection{\texorpdfstring{$\{0,1,2,3\}$}{\{0,1,2,3\}} graph classes with unbounded clique-width}

We now extend our results to graph classes $\GGG^\alpha$ where $\alpha$ is an infinite word over the alphabet $\{0,1,2,3\}$ containing an infinite number of letters from $\{2,3\}$. For this we will use the rank-width parameter described in Section \ref{cliquerank}. From \cite{collins:infinitely-many:} we have a toolkit of graph operations which we extend to show that the graph class $\GGG^\alpha$ contains a graph with a vertex-minor $H^\gamma_{1,1}(q,q)$ for some $q$ where $\gamma$ is an infinite word from the alphabet $\{2,3\}$. If we can make $q$ as large as we like then combining Lemma \ref{vminor} with Lemma \ref{lem-23unbound} gives us the result that $\GGG^\alpha$ has unbounded rank-width and therefore unbounded clique-width.

The following graph operations are demonstrated in \cite{collins:infinitely-many:} unless otherwise stated. Each operation takes a graph $H^\alpha_{1,1}(m,n)$ and uses local complementation and vertex deletion to create a $p \times q$ vertex-minor $H^\gamma_{1,1}(p,q)$, for some $p \le m$ and $q \le n$, where $\gamma$ is a $q-1$ letter word derived from $\alpha$ with certain letters removed.  We use the term  \emph{$0$ removal} where the removed letter(s) are $0$s and likewise \emph{$1$ removal} where the removed letter(s) are $1$s.

\textbf{$0$ Removal Operations}
\begin{enumerate}[label=(\roman*)]
	\item Removing a $0$ from the factor $00$ can be achieved by applying local complementation to each of the vertices in the middle column and then deleting the vertices in that column.
	\item Removing the $0$ from the factor $01$ can be achieved by applying local complementation to each of the vertices in the middle column sequentially from top to bottom and then deleting the vertices in that column. If the number of rows $m$ is even this is equivalent to removing the $0$ from the factor $01$. If the number of rows is odd the same result is achieved by modifying the process so that the local complementation ends on row $m-1$ and deleting the last row of vertices. The factor $10$ can be reduced to $1$ in the same way.
	\item Removing the $0$ from the factor $02$ can be achieved by local complementation on the vertices of the middle column and then deleting the odd rows. Also factor $20$ can be reduced to $2$, $03$ to $3$ and $30$ to $3$ in the same way. 
\end{enumerate} 

These operations allow us to create a vertex-minor $H^\gamma_{1,1}(p,q)$ with the letters of $\gamma$ from the alphabet $\{1,2,3\}$. 

\textbf{$1$ Removal Operations}
\begin{enumerate}[label=(\roman*)]
	\item Transforming the factor $211$ into $2$ can be achieved by using one pivot and deleting the first and last rows to give a $200$ factor and then using $0$ removal operations to reduce to $2$.  In the same way we can transform the factor $112$ into $2$, $311$ into $3$ and $113$ into $3$. 
	\item Transforming the factor $212$ into $22$ can be achieved by using one pivot and deleting the first and last rows to give a $202$ factor and then using $0$ removal operations to reduce to $22$ . In the same way we can also reduce $313$ to $33$. 
	\item Finally, we claim we can transform the factor $213$ into $22$. As this is not covered by \cite{collins:infinitely-many:} we give the proof here.
	
	Let $C_{k}$, $C_{k+1}$, $C_{k+2}$, $C_{k+3}$ be four consecutive columns of $H_{i,j}(m,n)$ such that $C_{k} \cup C_{k+1}$ induce a $2$-link, $C_{k+1} \cup C_{k+2}$ induce a $1$-link and $C_{k+2} \cup C_{k+3}$ induce a $3$-link.
	Let $x$ be the vertex in the first row of column $C_{k+1}$ and $y$ be the vertex in the last row of column $C_{k+2}$ It can be seen that by pivoting on the edge $xy$ and then deleting the whole of the first row, the second row to the right of column $C_{k+2}$ and deleting $y$ and vertices on the last row to the right of $y$, we have transformed the factor $213$ into $202$. We can then use the zero removal operations to reduce to $22$. In the same way we can also reduce $312$ to $33$.
	
\end{enumerate} 
\begin{obs}\label{obs4}
If $r$ is the number of rows prior to the removal of a $0$ or $1$ by one of the these operations then after the operation the number of rows left will be at least $(r/2)-2$. 
\end{obs}

Thus, by starting with a large enough number $m$ in our choice of $H_{i,j}(m,n)$, we may remove a finite number of $0$s and $1$s and still ensure that there are enough rows left at the end of the process.
 
We now have a complete set of tools, using local complementation and vertex removal applied to $H^\alpha_{1,1}(m,n)$, to create a vertex-minor $H^\gamma_{1,1}(p,q)$ with the letters of $\gamma$ coming only from the alphabet $\{2,3\}$.

\begin{lemma}\label{lem-0123unbound}
Let $\alpha$ be an infinite word over the alphabet $\{0,1,2,3\}$ which has an infinite number of letters from $\{2,3\}$. Further, let $\beta$ be a factor $\alpha_k \alpha_{k+1} \cdots \alpha_{k+p-1}$ of length $p$ which has $q$ ($0 < q \le p$) letters from $\{2,3\}$ and $p-q$ letters from $\{0,1\}$. 

Then $H^\alpha_{1,k}((q+4)2^{p-q}),p)$ has a vertex-minor isomorphic to $H^\gamma_{1,1}(q,q)$, for some word $\gamma$ using only letters from the alphabet $\{2,3\}$.
\end{lemma}

\begin{proof}
	This follows by applying the graph operations described above to remove the $0$s and $1$s. There are $p-q$ such letters, and using Observation~\ref{obs4} it can be seen that by starting with at least $(q+4)2^{p-q}$ rows, there will be at least $q$ rows remaining after executing the necessary operations to remove them.
\end{proof}

We now have the following theorem:

\begin{thm}\label{thm-0123unbound}
If $\alpha$ is an infinite word from the alphabet $\{0,1,2,3\}$ with an infinite number of non-zero letters, then the graph class $\GGG^\alpha$ is a hereditary class of graphs of unbounded clique-width.
\end{thm}
\begin{proof}
If there are no $2$s or $3$s or only a finite number of $2$s and $3$s in $\alpha$, we can use Lemma \ref{lem-01unbound}. 

If there is an infinite number of $2$s and $3$s in $\alpha$ then we can use Lemmas  \ref{lem-23unbound} and \ref{lem-0123unbound} as follows. For any $q$ we can find a graph $G$ in $\GGG^\alpha$ that has a vertex-minor $H^\gamma_{1,1}(q,q)$, for some infinite word $\gamma$ using only letters from the alphabet $\{2,3\}$. In turn, $H^\gamma_{1,1}(q,q)$ contains an induced subgraph $W^\gamma_{q/2}$, so using Lemma \ref{vminor} we have 
\[rwd(G)\ge rwd(H^\gamma_{1,1}(q,q)) \ge rwd (W^\gamma_{q/2}).\]
But from Lemma \ref{Walpha} $cwd(W^\gamma_{q/2})\ge q/8 \rightarrow \infty$  as $q \rightarrow \infty$ (and hence also $rwd(W^\gamma_{q/2}) \rightarrow \infty$) so it follows that $\GGG^\alpha$ is a graph class with unbounded rank-width and hence unbounded clique-width.
\end{proof}

%
%
%
%
%
%
\section{Minimality of almost periodic graph classes}\label{Sect:Minimal}


Let $\alpha$ be an infinite almost periodic word from the alphabet $\{0,1,2,3\}$, with at least one non-zero letter.
In this section we will prove that the graph classes $\GGG^\alpha$ are minimal of unbounded clique-width. To do this, we must show that any proper hereditary subclass has bounded clique-width. If $\CCC$ is a proper hereditary subclass of $\GGG^\alpha$ then there must exist a non-trivial finite forbidden graph $F$ that is in $\GGG^\alpha$ but not in $\CCC$. In turn, this graph $F$ must be an induced subgraph of some $H^\alpha_{1,j}(k,k)$ for some $k \ge 2$. 

Consider a graph $G \in \CCC \subseteq \Free(H^\alpha_{1,j}(k,k))$. If there exists an embedding $\phi:V(G) \rightarrow V(\PPP^\alpha)$  straddling columns $C_j,\dots, C_{j+k-1}$ of $\PPP^\alpha$ then there must be limits on the vertices of $\phi(V(G))$ in these columns to avoid creating an induced subgraph isomorphic to $H^\alpha_{1,j}(k,k)$.

Clearly the most obvious thing to avoid is any $k$ complete horizontal rows which would automatically generate the forbidden graph. For graphs only involving letters $0$ and $1$ this is sufficient, and was dealt with in \cite{collins:infinitely-many:}. However, letters $2$ and $3$ are more complex. Lozin in \cite{lozin:minimal-classes:} dealt with the class $\GGG^{2^\infty}$ by introducing the concept of \emph{clusters} and creating a cluster graph which provided a method of defining a partition of the vertices of $G$ which gave the desired boundary on clique-width.   

In considering $\{0,1,2,3\}$ graphs we use a modified version of the cluster graph method used in \cite{lozin:minimal-classes:} which we describe in Section \ref{cluster}.

The following lemma will be used to place a bound on the clique-width of induced subgraphs of $\PPP^\alpha$. Given a graph $G$ and a subset of vertices $U \subseteq V(G)$, 2 vertices of $U$ will be called \emph{U-similar} if they have the same neighbourhood outside $U$. $U$-similarity is an equivalence relation. The number of equivalence classes of $U$ in $G$ will be denoted $\mu_G(U)$ (or $\mu(U)$ when the context is clear). Also, by $G[U]$ we will denote the subgraph of $G$ induced by $U$.  It follows that $U$ is a module of $G$ if and only if $\mu(U)=1$.

\begin{lemma}[Lozin~\cite{lozin:minimal-classes:}] \label{lemma-partitions} 
If the vertices of a graph $G$ can be partitioned into subsets $U_1,U_2, \dots ,U_n$ in such a way that for every $i$
\begin{enumerate}[label=(\alph*)]
\item the clique-width of $G[U_i]$ is at most $k \ge 2$, and 
\item $\mu(U_i) \le l$ and $\mu(U_1 \cup \cdots \cup U_i) \le l$,
\end{enumerate}
then the clique-width of $G$ is at most $kl$.
\end{lemma}

\begin{cor} \label{cor-H-cwd}
Suppose for $n \ge 2$ that the $(n-1)$ letter factor $\alpha_j \alpha_{j+1} \cdots \alpha_{j+n-2}$  consists of $t \ge 1$ non-zeros and $n-t-1$ zeros. Then the clique-width of $H^\alpha_{i,j}(m,n)$, ($m \ge 1$), is at most $6t+3$.
\end{cor} 
\begin{proof}
To build $H^\alpha_{i,j}(m,n)$ we partition it into subsets $U_1,U_2, \dots ,U_m$ by including in $U_k$ the vertices of the $k$-th row of $H^\alpha_{i,j}(m,n)$. This means $G(U_k)$ is a disjoint union of paths so has clique-width at most $3$. Also, only a vertex in $U_k$ that is in a column $s$ such that $\alpha_{s-1}$ and/or $\alpha_{s}$ is non-zero could have a neighbour outside $U_k$, so $\mu(U_k)\le 2t+1$. Also, it is not difficult to see that $\mu(U_1 \cup \cdots \cup U_k) \le 2t+1$, as the vertices in each column have the same neighbourhood in $H^\alpha_{i,j}(m,n)$ outside $U_1 \cup \cdots \cup U_k$. Therefore, the result follows by applying Lemma \ref{lemma-partitions}. 
\end{proof}

We will also make  use of the following lemma. 
\begin{lemma}[{Courcelle and Olariu~\cite[Corollary $3.6$]{courcelle:upper-bounds-to:}}]\label{lemma-prime}
For every graph $G$, \[\cwd(G) = \max\{\cwd(H) \mid H \in \Prime(G)\}.\]
\end{lemma}

Thus we can assume that our arbitrary graph $G$ is prime, and therefore connected, since if it were not so, we could any prime subgraph $H$ which has the same clique-width.

%
%
%
%
%
\subsection{Cluster graphs}\label{cluster}

Consider a connected graph $G$ embedded in $\PPP^\alpha$ such that its leftmost vertex is in column $C_a$ and rightmost vertex in column $C_{a+n-1}$. Let a \emph{left module} of $G\cap C_j$ be a maximal set of  vertices in that column that are indistinguishable by vertices in $G\cap C_{j-1}$. Similarly, a \emph{right module} of $G\cap C_j$ is a maximal set of  vertices that are indistinguishable by vertices in $G\cap C_{j+1}$. Thus the vertices of every column of $G$, except the leftmost and rightmost columns, can be partitioned in two ways, as a set of left modules or as a set of right modules.

Now consider an $\alpha_j$-link, $G_j$ (the subgraph of $G$ induced by the vertices of $G\cap C_j$ and $G\cap C_{j+1}$). 
For convenience, we will say $G_j$ is in \emph{standard form} if, without changing the vertical order of the vertices, it is presented as an induced subgraph of $H^\alpha_{1,j}(m,2)$ with minimum $m$ (i.e.\ taking out any superfluous gaps).
Let $G_j^s$ be the standard form of $G_j$, noting that the (left and right) modules of $G_j$ and $G_j^s$ are identical.  Suppose $R$ is the set of right modules of $G_j^s\cap C_j$ and $L$ the set of left modules of $G_j^s\cap C_{j+1}$. We will say that two modules $A$ and $B$ \emph{overlap} if the set of rows containing vertices of $A$ has non-zero intersection with the set of rows containing vertices of $B$. It can be seen that a right module in $R$ can only overlap with a left module in $L$ on at most one row, for if they overlapped on two or more rows they would no longer be (right/left) modules. Furthermore, a right module in $R$ cannot overlap with more than one left module in $L$ and vice-versa.

If a right module in $R$ overlaps with a left module in $L$, the two modules can be paired to form a \emph{cluster}. This pairing process is well-defined and matches all such modules except for at most one unmatched right module in $G \cap C_j$ and one unmatched left module in $G \cap C_{j+1}$. These unmatched right/left modules have the characteristic they are indistinguishable to all vertices in the column to the right/left respectively. We will refer to them as right/left \emph{boundary modules} and the vertices in them as right/left \emph{boundary vertices} respectively. We put each boundary vertex in its own cluster in $G_j$. 
 
Hence, the clusters of $G_j$ form a partition of the vertices of the $\alpha_j$-link. If $\alpha_j$ is $0$ or $1$ then every cluster in $G_j$ is either a horizontal pair of vertices or a boundary vertex. If $\alpha_j$ is $2$ or $3$ then each cluster is a complete bipartite induced subgraph of $G_j$ or a boundary vertex. When there are no boundary vertices, $G_j$ is isomorphic to $H^\alpha_{1,j}(m,2)$ consisting of $m$ clusters, each containing two vertices of a same row. 

At this stage, the vertices in the leftmost ($G \cap C_a$) and rightmost ($G \cap C_{a+n-1}$) columns are each only in one cluster as they only appear in one $\alpha_j$-link. We now add two additional columns of clusters, one to the left of $G$ with a cluster for each vertex of $G \cap C_a$, and one to the right of $G$ with a cluster for each vertex of column $G \cap C_{a+n-1}$. Thus the vertices of every column $G\cap C_j$ of $G$ are now in two clusters.

With any finite induced subgraph $G$ of $\PPP^\alpha$ we associate an oriented graph which we call the \emph{cluster graph} $B(G)$, whose vertices are each associated with one of the clusters of  $G$.  The vertices of $B(G)$ representing clusters of the same $\alpha_j$-link, $G_j$, we call a column of $B(G)$, and denote this by $B(G_j)$. Of the two additional cluster columns defined in the previous paragraph we will call the one on the left $B(G_{a-1})$ and the one on the right $B(G_{a+n})$. For ease of exposition, we will always present the vertices of $B(G_j)$ in the same vertical order as in $G$.

In the following we denote the $i$-th cluster of $G_j$, counting from top to bottom, as $K_{i,j}$, with corresponding vertex, $u_{i,j}$ in $B(G_j)$. The edges of $B(G)$ are defined as follows.

\begin{description}
\item [Type A] If $K_{r,j}$ has a vertex of $G$ in common with cluster $K_{s,j+1}$ then $B(G)$ has a directed edge $(u_{r,j},u_{s,j+1})$ (i.e an edge oriented from $u_{r,j}$ to $u_{s,j+1}$). Note that if there was more than one vertex of $G$ in the intersection between two clusters of $G$, then these form a module of size greater than one and $G$ is not prime, a contradiction. Hence, any two clusters of $G$ have at most one vertex in the intersection. It follows that each Type A edge of $B(G)$ corresponds to a vertex of $G$. 
\item [Type B] Let $u_{i,j}$ and $u_{i+1,j}$ be two consecutive vertices in a column of $B(G)$. If $\alpha_j = 2$ then $B(G)$ has a directed edge $(u_{i,j},u_{i+1,j})$ and if $\alpha_j = 3$ then $B(G)$ has a directed edge $(u_{i+1,j},u_{i,j})$. 
\end{description}

Edges of type A are oriented from left to right and go between consecutive columns of $B(G)$ whilst edges of type B are oriented down when $\alpha_j = 2$ and up when $\alpha_j = 3$. Drawing $B(G)$ by arranging the vertices in columns in the same order as the respective clusters of $G$ it becomes clear that $B(G)$ is a planar graph.

If we have a right boundary vertex in $C_j$ then it will be in both a cluster in $G_{j-1}$, say $K_{r,j-1}$, and a singleton cluster in $G_{j}$, say $K_{s,j}$.  The two vertices, $u_{r,j-1}$ and $u_{s,j}$ in $B(G)$ associated with these clusters will be joined by a type A edge. However, there can be no type A edge to the immediate right of $u_{s,j}$ as $K_{s,j}$ contains no vertex from column $C_{j+1}$. Similarly for left boundary vertices, mutatis mutandis.

%
%

\begin{figure}\centering
{\centering
\begin{tikzpicture}[scale=0.65]
\foreach \i in {1,...,7}
	\foreach \j in {1,...,7}
		\node[fill=black!20,draw=black!20] at (\i,\j) {};
\foreach \jset [count=\i] in {{1,2,3,5,6}, {1,2,3,5,6,7}, {1,2,3,4,5,7},{1,2,4,5,7},{1,3,4,5,6,7},{1,2,3,4,6,7},{1,2,3,5,6,7}}
 	\foreach \j in \jset {
 		\node (\i-\j) at (\i,\j) {};
 	}
\foreach \j in {1,2,3,5,6}
	\draw (1-\j) -- (2-\j);
\foreach \j/\kset in {1/{2,3,4,5,7},2/{1,3,4,5,7},3/{1,2,4,5,7},5/{1,2,3,4,7},6/{1,2,3,4,5,7},7/{1,2,3,4,5}}
	\foreach \k in \kset
		\draw (2-\j) -- (3-\k);
\foreach \j in {1,2,4,5,7}
	\draw (3-\j) -- (4-\j);
\foreach \j/\kset in {1/{1},2/{1},4/{1,3,4},5/{1,3,4,5},7/{1,3,4,5,6,7}} 
	\foreach \k in \kset
		\draw (4-\j) -- (5-\k);
\foreach \j/\kset in {1/{1,2,3,4,6,7},3/{3,4,6,7},4/{4,6,7},5/{6,7},6/{6,7},7/{7}} 
	\foreach \k in \kset
		\draw (5-\j) -- (6-\k);
\foreach \j/\kset in {1/{1},2/{1,2},3/{1,2,3},4/{1,2,3},6/{1,2,3,5,6},7/{1,2,3,5,6,7}} 
	\foreach \k in \kset
		\draw (6-\j) -- (7-\k);
\begin{scope}[shift={(8,0)}]
\foreach \jset [count=\i] in {{1,2,3,5,6}, {1,2,3,5,6,7}, {1,2,3,4,5,7},{1,2,4,5,7},{1,3,4,5,6,7},{1,2,3,4,6,7},{1,2,3,5,6,7}}
 	\foreach \j in \jset {
 		\node (\i-\j) at (\i,\j) {};
 	}
\foreach \j in {1,2,3,5,6}
	\draw (1-\j) -- (2-\j);
\foreach \j/\kset in {1/{2,3,4,5,7},2/{1,3,4,5,7},3/{1,2,4,5,7},5/{1,2,3,4,7},6/{1,2,3,4,5,7},7/{1,2,3,4,5}}
	\foreach \k in \kset
		\draw (2-\j) -- (3-\k);
\foreach \j in {1,2,4,5,7}
	\draw (3-\j) -- (4-\j);
\foreach \j/\kset in {1/{1},2/{1},4/{1,3,4},5/{1,3,4,5},7/{1,3,4,5,6,7}} 
	\foreach \k in \kset
		\draw (4-\j) -- (5-\k);
\foreach \j/\kset in {1/{1,2,3,4,6,7},3/{3,4,6,7},4/{4,6,7},5/{6,7},6/{6,7},7/{7}} 
	\foreach \k in \kset
		\draw (5-\j) -- (6-\k);
\foreach \j/\kset in {1/{1},2/{1,2},3/{1,2,3},4/{1,2,3},6/{1,2,3,5,6},7/{1,2,3,5,6,7}} 
	\foreach \k in \kset
		\draw (6-\j) -- (7-\k);
\tikzstyle{cvxpath}=[thick,fill,fill opacity=0.2];
\begin{pgfonlayer}{background}
\draw[green!50!black,cvxpath] (1-1) circle (0.2cm);
\draw[green!50!black,cvxpath] (1-2) circle (0.2cm);
\draw[green!50!black,cvxpath] (1-3) circle (0.2cm);
\draw[green!50!black,cvxpath] (1-5) circle (0.2cm);
\draw[green!50!black,cvxpath] (1-6) circle (0.2cm);
\draw[red,cvxpath] (2-7) circle (0.2cm);
\draw[red,cvxpath] \convexpath{1-1,2-1}{0.2cm};
\draw[red,cvxpath] \convexpath{1-2,2-2}{0.2cm};
\draw[red,cvxpath] \convexpath{1-3,2-3}{0.2cm};
\draw[red,cvxpath] \convexpath{1-5,2-5}{0.2cm};
\draw[red,cvxpath] \convexpath{1-6,2-6}{0.2cm};
\draw[blue,cvxpath] (2-3) circle (0.2cm);
\draw[blue,cvxpath] (2-6) circle (0.2cm);
\draw[blue,cvxpath] (3-4) circle (0.2cm);
\draw[blue,cvxpath] \convexpath{2-1,3-1}{0.2cm};
\draw[blue,cvxpath] \convexpath{2-2,3-2}{0.2cm};
\draw[blue,cvxpath] \convexpath{2-3,3-3}{0.2cm};
\draw[blue,cvxpath] \convexpath{2-5,3-5}{0.2cm};
\draw[blue,cvxpath] \convexpath{2-7,3-7}{0.2cm};
\draw[green!50!black,cvxpath] (3-3) circle (0.2cm);
\draw[green!50!black,cvxpath] \convexpath{3-1,4-1}{0.2cm};
\draw[green!50!black,cvxpath] \convexpath{3-2,4-2}{0.2cm};
\draw[green!50!black,cvxpath] \convexpath{3-4,4-4}{0.2cm};
\draw[green!50!black,cvxpath] \convexpath{3-5,4-5}{0.2cm};
\draw[green!50!black,cvxpath] \convexpath{3-7,4-7}{0.2cm};
\draw[red,cvxpath] \convexpath{4-2,5-1,4-1}{0.2cm};
\draw[red,cvxpath] \convexpath{4-5,5-5}{0.2cm};
\draw[red,cvxpath] \convexpath{4-4,5-4,5-3}{0.2cm};
\draw[red,cvxpath] \convexpath{4-7,5-7,5-6}{0.2cm};
\draw[blue,cvxpath] \convexpath{5-1,6-2,6-1}{0.2cm};
\draw[blue,cvxpath] \convexpath{5-3,6-3}{0.2cm};
\draw[blue,cvxpath] \convexpath{5-4,6-4}{0.2cm};
\draw[blue,cvxpath] \convexpath{5-5,5-6,6-6}{0.2cm};
\draw[blue,cvxpath] \convexpath{5-7,6-7}{0.2cm};
\draw[green!50!black,cvxpath] \convexpath{6-1,7-1}{0.2cm};
\draw[green!50!black,cvxpath] \convexpath{6-2,7-2}{0.2cm};
\draw[green!50!black,cvxpath] \convexpath{6-3,6-4,7-3}{0.2cm};
\draw[green!50!black,cvxpath] \convexpath{6-6,7-6,7-5}{0.2cm};
\draw[green!50!black,cvxpath] \convexpath{6-7,7-7}{0.2cm};
\draw[red,cvxpath] (7-1) circle (0.2cm);
\draw[red,cvxpath] (7-2) circle (0.2cm);
\draw[red,cvxpath] (7-3) circle (0.2cm);
\draw[red,cvxpath] (7-5) circle (0.2cm);
\draw[red,cvxpath] (7-6) circle (0.2cm);
\draw[red,cvxpath] (7-7) circle (0.2cm);
\end{pgfonlayer}
\end{scope}
\begin{scope}[shift={(16.5,0)}]
\foreach \j in {1,2,3,5,6}
	\node[green!50!black] (0-\j) at (0.5,\j) {};
\foreach \j in {1,2,3,5,6,7}
	\node[red] (1-\j) at (1.5,\j) {};
\foreach \j in {1,2,3,4,5,6,7}
	\node[blue] (2-\j) at (2.5,\j) {};
\foreach \j in {1,2,3,4,5,7}
	\node[green!50!black] (3-\j) at (3.5,\j) {};
\foreach \j [count=\i] in {1.5,3.5,5,6.5}
	\node[red] (4-\i) at (4.5,\j) {};
\foreach \j [count=\i] in {1.5,3,4,5.5,7}
	\node[blue] (5-\i) at (5.5,\j) {};
\foreach \j [count=\i] in {1,2,3.5,5.5,7}
	\node[green!50!black] (6-\i) at (6.5,\j) {};
\foreach \j in {1,2,3,5,6,7}
	\node[red] (7-\j) at (7.5,\j) {};
	\begin{scope}[decoration={
    	markings,
    	mark=at position 0.5 with {\arrow{>}}}] 
	\foreach \j in {1,2,3,5,6}
		\draw[postaction={decorate}] (0-\j) -- (1-\j);
	\foreach \j in {1,2,3,5,6,7}
		\draw[postaction={decorate}] (1-\j) -- (2-\j);
	\foreach \j in {1,2,3,4,5,7}
		\draw[postaction={decorate}] (2-\j) -- (3-\j);
	\foreach \i/\j in {1/1,2/1,4/2,5/3,7/4}
		\draw[postaction={decorate}] (3-\i) -- (4-\j);
	\foreach \i/\j in {1/1,2/2,2/3,3/4,4/4,4/5}
		\draw[postaction={decorate}] (4-\i) -- (5-\j);
	\foreach \i/\j in {1/1,1/2,2/3,3/3,4/4,5/5}
		\draw[postaction={decorate}] (5-\i) -- (6-\j);
	\foreach \i/\j in {1/1,2/2,3/3,4/5,4/6,5/7}
		\draw[postaction={decorate}] (6-\i) -- (7-\j);
	\foreach \i/\j in {1/2,2/3,3/4}
		\draw[dotted,postaction={decorate}] (4-\j) -- (4-\i);
	\foreach \i/\j in {1/2,2/3,3/4,4/5} {%
		\draw[dotted,postaction={decorate}] (5-\i) -- (5-\j);
		\draw[dotted,postaction={decorate}] (6-\j) -- (6-\i);
	}
	\end{scope}
\end{scope}
\end{tikzpicture}
\par}

\caption{For $\alpha=010232\cdots$, an illustration of the formation of a cluster graph. From left to right: an embedding of $G^*$ in $H_{1,1}(7,7)$, the clusters associated with this embedding, and the cluster graph $B^*=B(G^*)$, with type B edges indicated by dotted lines.}
	\label{fig:cluster}

\end{figure}
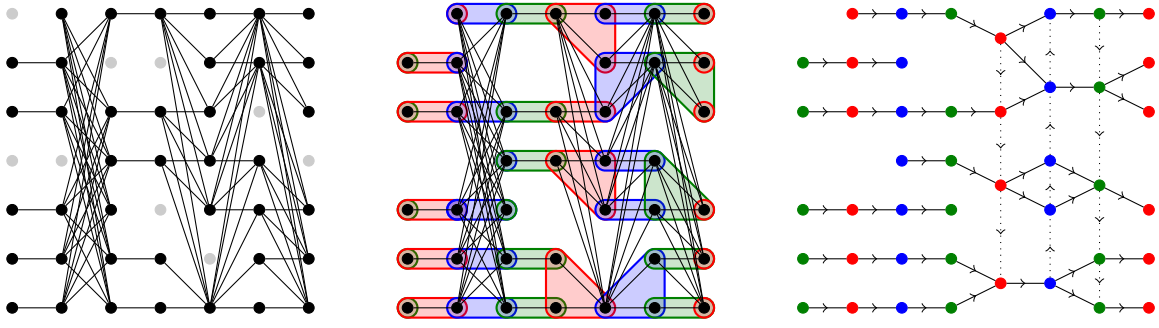
%
%

An example of a cluster graph is shown in Figure \ref{fig:cluster}.

As we assume $G$ is prime (from Lemma \ref{lemma-prime}) and therefore connected, when it is embedded in $\PPP^\alpha$ it must occupy vertices in consecutive columns (i.e. if it has one or more vertices in columns $C_x$ and $C_{x+2}$ then it must also have at least one vertex in column $C_{x+1}$). Suppose $G$ straddles a set of columns including the $k$ columns $C_j, \dots, C_{j+k-1}$ with defining factor $\beta=\alpha_{j} \alpha_{j+1}\cdots \alpha_{j+k-2}$. Let us denote the subgraph of $G$ induced by these columns $G^*$. The respective graph $B(G^*)$ will be denoted by $B^*$; it has $k+1$ columns denoted by $B_1, \dots, B_{k+1}$. It can be seen that if $B^*$ has k directed disjoint paths from column $B_1$ to column $B_{k+1}$ then $G^*$ contains the forbidden subgraph $H^\alpha_{1,1}(k,k)$, although the reverse is not necessarily true.

This leads to the following result:
\begin{lemma}\label{lemma-paths}
Let $\CCC$ be a proper subclass of $\GGG^\alpha$ such that $\CCC \subseteq \Free(H^\alpha_{1,j}(k,k))\subset \GGG^\alpha$. Furthermore, let $G$ be any graph in $\CCC$ with  induced subgraph $G^*$ and associated cluster graph $B^*$ defined as above. Then $B^*$ can have at most $k-1$ directed paths from column $B_1$ to column $B_{k+1}$.
\end{lemma}

%
%
%
%
%
\subsection{Applying Menger's Theorem}\label{mengthm}

We will be using Menger's Theorem to help us define a partition of the vertices of $G$ on which to apply Lemma \ref{lemma-partitions}. Menger's Theorem is one of the cornerstones of graph theory.

\begin{thm}[Menger, see e.g.\ Diestel~\cite{diestel:graph-theory5:}] \label{Menger} 
Let $G$ be a directed graph and $A,B \subseteq V(G)$. Then the minimum number $n$ of vertices separating $A$ from $B$ in $G$ is equal to the maximum number of disjoint directed $A \rightarrow B$ paths in $G$.
\end{thm} 

\begin{cor} \label{cor-Menger}
If $S$ is a set of $n$ vertices forming such a separator, then  there exists a partition of $V(G) \setminus S$ into two sets $X$ and $Y$ such that there are no edges directed from a vertex in $X$ to a vertex in $Y$.  
\end{cor} 
\begin{proof}
Let $X$ denote the set of vertices $V(G) \setminus S$ that can be reached from $A$ by following directed paths, and let $Y$ denote the set of vertices of $V(G) \setminus S$ from which there starts a directed path that ends in a vertex of $B$. Now, $X$ and $Y$ are disjoint (by Menger's theorem). If there are any vertices of $V(G) \setminus S$ that lie in neither $X$ nor $Y$, we can assign them to either arbitrarily.  Now every edge with one endpoint in $X$ and the other in $Y$ must be oriented from $Y$ to $X$, otherwise we find a directed path from $A$ to $B$.
\end{proof}

We can apply this to the cluster graph $B^*$  referred to in Lemma \ref{lemma-paths} with columns $B_1$ and $B_{k+1}$ connected to each other by a set $P$ of at most $k-1$ disjoint paths. Denote $s = |P|$. The $s$ paths of $P$ cut $B^*$ into $s+1$ horizontal \emph{stripes}, that is, subgraphs induced by two consecutive paths in $P$ and all the vertices between them ($s-1$ such stripes) and $2$ further stripes for the subgraphs induced by the top path and all vertices above it, and the bottom path and all vertices below it.  

From Menger's Theorem these two columns can be separated from each other by a set $S$ of $s \le k-1$ vertices, containing exactly one vertex in each of the paths, such that there are no paths from $B_1$ to $B_{k+1}$ that avoid this set $S$. From Corollary \ref{cor-Menger} we have a partition of the vertices of $B^* \setminus S$ into two sets $X_{V(B^*)}$ and $Y_{V(B^*)}$ such that there are no directed edges from a vertex in $X_{V(B^*)}$ to a vertex in $Y_{V(B^*)}$.  As $B^*$ is planar this means we can draw a curve $\Omega$ between $X_{V(B^*)}$ and $Y_{V(B^*)}$ such that this curve crosses $B^*$ at precisely the set $S$ and such that there are no directed edges crossing $\Omega$ from $X_{V(B^*)}$ to $Y_{V(B^*)}$.  It follows that we can partition the Type A edges of $B^*$ into two sets $X_{E(B^*)}$ and $Y_{E(B^*)}$ either side of $\Omega$, and as these Type A edges correspond to the vertices of $G^*$ then we also have a partition of these vertices into two sets $X_{V(G^*)}$ and $Y_{V(G^*)}$.

%
%
%
%
%
\subsection{Almost periodic \texorpdfstring{$\{0,1,2,3\}$}{\{0,1,2,3\}} graph classes are minimal of unbounded clique-width}

We now come to the key result of Section \ref{Sect:Minimal}.

\begin{lemma}\label{lemma-minim}
	Let $\alpha$ be an infinite almost periodic word from the alphabet $\{0,1,2,3\}$ which has at least one non-zero letter, and $k$ a natural number at least $2$. Further, let $\beta=\alpha_j\alpha_j+1\cdots\alpha_{j+k-2}$ be a $k-1$ letter factor  of $\alpha$ such that $\beta$ appears in every factor of $\alpha$ of length $\mathcal{L}(\beta)$, so that $H^\alpha_{1,j}(k,k)$ is a graph in $\GGG^\alpha$ whose edges correspond to the subword $\beta$. Then any graph $G$ in $\GGG^\alpha$ that is $H^\alpha_{1,j}(k,k)$-free has clique-width bounded by a constant $c(k,\mathcal{L}(\beta))$ that depends only on $k$ and $\mathcal{L}(\beta)$.
\end{lemma}

\begin{proof}
Let $G$ be a graph in $\GGG^\alpha$ that is $H^\alpha_{1,j}(k,k)$-free. In the following we refer to vertex grid coordinates $(x,y)$ of an embedding of $G$ in $\PPP^\alpha$ as described in Section~\ref{graphdef}. As before, we assume $G$ is prime (from Lemma~\ref{lemma-prime}) and therefore connected.

We define a partition $\{V_1, \dots, V_n\}$ of the vertices of $G$ as follows. Let $a$ be the first column of $\PPP^\alpha$ in which a vertex of $G$ is embedded. Denoting the set of vertices of $G$ in a set of consecutive columns as a \emph{bar}, let $V_i$ be the bar of $G$ in columns $[a+(i-1)(\mathcal{L}(\beta)+1)]$ through to $[(a-1)+i(\mathcal{L}(\beta)+1)]$.

The corresponding subword for the graph induced by bar $V_i$ is of length $\mathcal{L}(\beta)$ so must contain a copy of $\beta$ by definition. Let this copy of $\beta$ correspond to columns $C_y$, \dots, $C_{y+k-1}$ of $\PPP^\alpha$. Following the same notation as Section \ref{cluster} we define $G^*$ as the subgraph of $G$ induced by the columns $C_y$, \dots, $C_{y+k-1}$ and $B^*$ its respective cluster graph, with columns $B_1, \dots, B_{k+1}$. We define $P$, $S$, $s$, $\Omega$, $X_{E(B^*)}$, $Y_{E(B^*)}$, $X_{V(G^*)}$ and $Y_{V(G^*)}$ as in Section \ref{mengthm}.

We now show that the partition $X_{V(G^*)}/Y_{V(G^*)}$ of the vertices of $G^*$ defined in Section \ref{mengthm} gives us a number of equivalence classes, $\mu_{G^*}(X_{V(G^*)})$ and $\mu_{G^*}(Y_{V(G^*)})$, bounded by a function of $k$. We consider this in $3$ cases depending on the alphabet of $\beta$:

\begin{description}
\item [Case 1] $\beta$ is a subword from the alphabet $\{0,1\}$.

A $\{0,1\}$ cluster graph $B(G)$ contains only edges of type A and only horizontal paths. Every cluster is either a horizontal pair of vertices or a boundary vertex. Each row of $B(G)$ is either a (left to right) directed path or a disjoint union of directed paths. If a row is a disjoint union of paths then the gaps between the paths have either no vertex or a boundary vertex immediately on either side. 

It is easy to see that the curve $\Omega$ must traverse each stripe of $B^*$ by passing through a gap in each row between the paths at the top and bottom of the stripe.

From Section \ref{mengthm} the $X_{E(B^*)}/Y_{E(B^*)}$ partition of the Type A edges of $B^*$ defined by $\Omega$ gives a corresponding $X_{V(G^*)}/Y_{V(G^*)}$ partition of the vertices of $G^*$. We can partition the edges of $B^*$ that correspond to vertices of a column $C_j$ of $G^*$, into at most $2s+1$ subsets $C_{1,j}, \dots, C_{2s+1,j}$, as follows:
 
\begin{enumerate}[label=(\roman*)]
	\item The edges forming the paths of $P$ ($s$ edges/subsets).
	\item The remaining edges in each stripe (at most $s+1$ subsets). 
\end{enumerate}  
 
We claim that no vertex of $Y_{V(G^*)}$ can distinguish the vertices of $C_{i,j} \cap X_{V(G^*)}$. Suppose to the contrary, that a vertex $y \in Y_{V(G^*)}$ is not adjacent to $x_1 \in C_{i,j} \cap X_{V(G^*)}$ but is adjacent to $x_2 \in C_{i,j} \cap X_{V(G^*)}$. Then $x_1$ and $x_2$ cannot be in the same cluster in $G_j$, because they are on different rows, but one of them must be in the same cluster as $y$. But this cluster is then not a boundary cluster as it contains two vertices and hence the cluster must be on a path of $P$. But $x_1$ and $x_2$ are in different clusters in $G_j$ so cannot both be on the same path of $P$ and so are in different subsets, $C_{i,j} \cap X_{V(G^*)}$, a contradiction.
 
So the maximum number of equivalence classes in $C_j \cap X_{V(G^*)}$ is $2s+1 \le 2k-1$ and hence $\mu_{G^*}(X_{V(G^*)})$ is at most the number of different $C_{i,j}$'s, which is at most $k(2s+1) \le 2k^2-k$. Symmetrically, $\mu_{G^*}(Y_{V(G^*)}) \le 2k^2-k$.

\item [Case 2] $\beta$ is a subword from the alphabet $\{2,3\}$.

Without loss of generality, we may assume that no $\alpha$-link $G_j$ , where $\alpha_j \in \{2,3\}$, contains a boundary vertex. For if such vertices exist, they will be positioned at one extreme (top or bottom) of a column. It is then possible to add an additional vertex in the opposite column to turn them into a  cluster.  Therefore, by adding at most two vertices to each column of $G$, we can extend it to a graph $G^\prime$ which has no boundary vertices, contains $G$ as an induced subgraph and is $H^\alpha_{1,j}(k+2,k)$-free.

\begin{obs}\label{obs5}
The curve $\Omega$ traverses each stripe of $B^*$ monotonically in a horizontal direction, meaning that its x-coordinate changes within a stripe either non-increasingly or non-decreasingly. 
\end{obs}
\begin{proof}[Proof of Observation]
Suppose for a contradiction that $\Omega$ had an unavoidable local maximum within a stripe, we would have a vertex $v$ (to the left of the curve) that causes this maximum $x$-coordinate. Obviously, $v$ does not belong to $B_{k+1}$ (since otherwise $B_{k+1}$ is not separated from $B_1$), and $v$ must have a neighbour to its right within the stripe (since there are no boundary vertices in $G$). But then the Type A edge connecting $v$ to that neighbour would cross $\Omega$, which contradicts Corollary \ref{cor-Menger}. 
\end{proof}
This observation allows us to conclude that whenever $\Omega$ separates the Type A edges between two columns of $B^*$ within a stripe, the result is two intervals, one above $\Omega$ and one below it.

From Section \ref{mengthm} the $X_{E(B^*)}/Y_{E(B^*)}$ partition of the Type A edges of $B^*$ defined by $\Omega$ gives a corresponding $X_{V(G^*)}/Y_{V(G^*)}$ partition of the vertices of $G^*$. We can partition the Type A edges of $B^*$ that correspond to vertices of a column $C_j$ of $G^*$, into at most $4s+1 \le 4k-3$ subsets $C_{1,j}, \dots, C_{4s+1,j}$, as follow:

\begin{enumerate}[label=(\roman*)]
	\item The Type A edges intersecting the paths of $P$ ($s$ edges/subsets).
	\item For each such edge $e$, the Type A edges that have a common vertex with $e$, at most 2 subsets in each stripe (up to $2s$ subsets). 
	\item The remaining Type A edges in each stripe (at most $s+1$ subsets). 
\end{enumerate} 
 
From Observation \ref{obs5} the vertices of each $C_{i,j}$ form an interval, i.e., they are consecutive in $C_j$. We claim that no vertex of $Y_{V(G^*)}$ can distinguish the vertices of $C_{i,j} \cap X_{V(G^*)}$. Suppose to the contrary, that a vertex $y \in Y_{V(G^*)}$ is not adjacent to $x_1 \in C_{i,j} \cap X_{V(G^*)}$ but is adjacent to $x_2 \in C_{i,j} \cap X_{V(G^*)}$. Without loss of generality we assume that $\alpha_j=2$, as the case $\alpha_j=3$ follows by symmetry. 
 
The vertices $x_1$ and $x_2$ cannot be in the same cluster in $G_{j}$ because they are distinguished by $y$. Furthermore, as $\alpha_j=2$ then $y$ must be in column $C_{j+1}$ on a row above that of $x_1$ but below or level with the row of $x_2$.  We can assume that $y$ is in the same $G_j$ cluster as $x_2$ because if it is not, then it must be in a cluster positioned between the cluster containing $x_2$ and the cluster containing $x_1$. As $C_{i,j}$ is an interval, this cluster must include some other vertex $x_3 \in C_{i,j} \cap X_{V(G^*)}$, and we can proceed using $x_3$ instead of $x_2$. Let the cluster of $G_j$ including $y$ and $x_2$ be denoted $K_{r,j}$.
 
Let $u_{r,j}$ denote the vertex of $B^*$ corresponding to $K_{r,j}$. Also, let $e_{x_{1}}, e_{x_{2}}, e_y$ be the edges of $B^*$ corresponding to vertices $x_1$, $x_2$, and $y$ respectively. Since $e_{x_{2}}$ and $e_y$ are incident to $u_{r,j}$ but separated by $\Omega$, vertex $u_{r,j}$ lies on $\Omega$ and hence belongs to the separator $S$. Therefore, $u_{r,j}$ belongs to a path from $P$. But then $C_{i,j}$ is of the second type and therefore $e_{x_{1}}$ must also be incident to $u_{r,j}$. This contradicts the fact that $x_1$ does not belong to $K_{r,j}$. This contradiction shows that any two vertices of $C_{i,j} \cap X_{V(G^*)}$ have the same neighbourhood in $Y$.

So $\mu_{G^*}(X_{V(G^*)})$ is at most the number of different $C_{i,j}$s, which is at most $k(4s+1) \le 4k^2-3k$. Symmetrically, $\mu_{G^*}(Y_{V(G^*)}) \le 4k^2-3k$. 

\item [Case 3] $\beta$ is a subword from the alphabet $\{0,1,2,3\}$.

Each column $B_i$ of cluster graph $B^*$ is associated with a letter of $\beta$.   We can divide $B^*$ into alternating $\{0,1\}$ bars and $\{2,3\}$ bars (reminder, a bar is a set of consecutive columns). Suppose we label these bars $D_1, D_2, \dots, D_m$ of lengths $k_1, k_2, \dots, k_m$  so that $k_1+k_2+\cdots+k_m=k-1$. Without loss of generality, we will say that if $i$ is odd, $D_i$ is a $\{0,1\}$ bar and if $i$ is even $D_i$ is a $\{2,3\}$ bar.

Define the partition curve $\Omega$ as before. If $\Omega$ stays in only one $\{0,1\}$ or $\{2,3\}$ bar then we can revert to Case 1 or Case 2. If $\Omega$ straddles several $\{0,1\}$ and $\{2,3\}$ bars we can argue as follows.

It can be observed that, within each stripe, $\Omega$ can only pass at most once through each bar in $B^*$. For if it passed twice through a column in a $\{2,3\}$ bar, in a given stripe, with at least one vertex between the two sections of $\Omega$ then there must be a Type B edge passing from $X_{V(B^*)}$ to $Y_{V(B^*)}$ which contradicts Corollary \ref{cor-Menger} of Menger's Theorem.

Also, for the same reasons given in Case 2, within each stripe, the line $\Omega$ must pass across $\{2,3\}$ bars monotonically in a left/right x-coordinate sense.

Using the same arguments as used in the $\{0,1\}$ and $\{2,3\}$ proofs we can partition the vertices of $G^*$ into sets $X_{V(G^*)}$ and $Y_{V(G^*)}$ such that the maximum number of different equivalence classes for a column $X_{V(G^*)} \cap C_j$  is $(2k-1)\le (4k-3)$ for a $\{0,1\}$ column and $(4k-3)$ for a $\{2,3\}$ column. So for $k \ge 2$ we have $\mu_{G^*}(X_{V(G^*)}) \le 4k^2-3k$. Symmetrically, $\mu_{G^*}(Y_{V(G^*)}) \le 4k^2-3k$. 

\end{description}

Using this $X_{V(G^*)}/Y_{V(G^*)}$ partition of the vertices of $G^*$ we can create a partition of $V_i$. All vertices in $V_i$ in columns to the left of $G^*$ are added to the vertices of $X_{V(G^*)}$ and all the vertices in $V_i$ in columns to the right of $G^*$  are added to the vertices of $Y_{V(G^*)}$ to produce a partition of bar $V_i$ into two parts $X_i$ and $Y_i$. Let $U_i = Y_{i-1} \cup X_i$. Each $G[U_i]$ has at most $2(\mathcal{L}(\beta)+1)$ columns. The subsets $U_1$, \dots $U_n$ form a partition of the vertices of $G$, such that for every $i$:
\begin{enumerate}[label=(\alph*)]
\item using Corollary \ref{cor-H-cwd} the clique-width of $G[U_i]$ is at most $6(2\mathcal{L}(\beta)+2)+3 = 12\mathcal{L}(\beta)+15$, and 
\item $\mu(U_i) \le 2(4k^2-3k)$ and $\mu(U_1 \cup \cdots \cup U_i) \le 2(4k^2-3k)$.
\end{enumerate}
Thus from Lemma~\ref{lemma-partitions} the clique-width of $G$ is at most $(12\mathcal{L}(\beta)+15)(8k^2-6k)$.

Hence the clique-width of $G$ is bounded by a constant that depends only on $k$ and $\mathcal{L}(\beta)$.
\end{proof}

\begin{thm}\label{thm-0123minim}
Let $\alpha$ be an infinite almost periodic word over the alphabet $\{0,1,2,3\}$ containing at least one non-zero letter. Then the class $\GGG^\alpha$ is a minimal hereditary class of graphs of unbounded clique-width.
\end{thm}

\begin{proof}
If $\CCC$ is a proper hereditary subclass of $\GGG^\alpha$ then there must exist a non-trivial finite forbidden graph $F$ that is in $\GGG^\alpha$ but not in $\CCC$. But $F$ must be an induced subgraph of some $H^\alpha_{1,j}(k,k)$ so $\CCC\subseteq \Free(H^\alpha_{1,j}(k,k)$ and Lemma~\ref{lemma-minim} gives us a bound on the clique-width. Hence, $\GGG^\alpha$ is a minimal hereditary class of graphs of unbounded clique-width.  
\end{proof}

\subsection{Uncountably many minimal graph classes with unbounded clique-width}

We now proceed to show that there is an uncountably infinite number of such graph classes. To do this we will use the class of almost periodic sequences known as \emph{Sturmian}. One definition of a Sturmian sequence is a binary sequence that has \emph{complexity} $p_\alpha(n)=n+1$, where the complexity function $p_\alpha(n)$ is the number of different factors of length $n$ in $\alpha$ \cite{fogg:substitutions:}. 

An alternative characterisation of Sturmian sequences is as rotation sequences defined by an irrational number, and hence it follows that the number of such sequences is uncountably infinite. We say that two sequences are \emph{locally isomorphic} if they have the same factors. If two Sturmian sequences are locally isomorphic this means they have the same $n+1$ factors of length $n$ out of a possible $2^n$ such factors \cite{wen:local-isom:}. Hence the set of Sturmian sequences with a particular set of factors is countable in number and so it follows there is an uncountable number of such sets with different factors.

We denote \emph{rev($\beta$)} as the sequence $\beta$ in reverse order (mirror image).

\begin{lemma}\label{lem-embed}
Let $\alpha$ be an infinite binary word and $\beta$ a finite binary word of length $k-1$ ($k \ge 2$) with at least one $1$. Further, let $F_\beta$ be the graph $H^\beta_{1,1}(3,k)$.

Then $F_\beta$ can be embedded in $\PPP^\alpha$ if and only if $\beta$ or rev($\beta$) is a factor of $\alpha$.

Furthermore, if $k \ge 3$, such embedding is only possible in $3$ rows and $k$ consecutive columns $s,\dots, s+k-1$ when $\beta$ (or rev($\beta$)) $=\alpha_s \cdots \alpha_{s+k-1}$.
\end{lemma}

\begin{proof}
Clearly, by its definition, $F_\beta$ can be embedded in $\PPP^\alpha$ in the way described if $\beta$ (or rev($\beta$)) is a factor of $\alpha$. [To avoid much repetition in what follows we will just refer to $\beta$ to mean $\beta$ or rev($\beta$).] We prove that $F_\beta$ can only be embedded in $\PPP^\alpha$ in the way described, and only if $\beta$ is a factor of $\alpha$, by induction on $k$. 

Firstly, if $k=2$ then $\beta=1$ and $F_\beta=C_6$, the cycle on $6$ vertices. It is trivial to see that this can be embedded in $\PPP^\alpha$ only if there is at least one $1$ in $\alpha$. In fact, $C_6$ can be embedded in two ways. Firstly, (Method $1$) in the way described in the Lemma, with $6$ vertices from $3$ rows and $2$ consecutive columns, or secondly, (Method $2$) over $4$ consecutive columns with $1$ vertex from the first and last column and $2$ from each of the middle columns.

If $k=3$ then $\beta$ must be $10$, $01$ or $11$. $F_\beta$ still includes an induced subgraph $C_6$, but now the addition of 3 more vertices and corresponding edges means we can no longer use Method $2$. Hence, $F_\beta$ can only be embedded in $\PPP^\alpha$ in the way described in the Lemma (Method $1$) .

Next using the strong induction hypothesis, we assume that the Lemma is true for all words of length less than $k-1$. Thus if $\beta$ contains a factor that is not a factor of $\alpha$ then $F_\beta$ cannot embed in $\PPP^\alpha$. 

If $F_\beta$ does embed in $\PPP^\alpha$ then, if $\beta^-$ is the word $\beta$ without its last letter, we must have $\beta^-$ a factor of $\alpha$ where $F_{\beta^-}$ can only embed in $\PPP^\alpha$ by Method 1. Now it is straightforward to see that this cannot be extended to $F_\beta$ if the next letter is not the same as the last letter of $\beta$,  and that if it is the same, it can only be done by Method 1.
\end{proof}

\begin{thm} \label{cor-uncountable}
There exists an uncountably infinite number of minimal hereditary classes of graphs of unbounded clique-width.
\end{thm}

\begin{proof}
There exists an uncountably infinite number of Sturmian binary sequences that are not locally isomorphic. Suppose we have Sturmian words $\alpha_1$ and $\alpha_2$ that have unique factors $\beta_1$ and $\beta_2$ respectively. Then using Lemma \ref{lem-embed}, the class $\GGG^{\alpha_1}$ does not contain the graph $F_{\beta_2}$ and the class $\GGG^{\alpha_2}$ does not contain the graph $F_{\beta_1}$. So $\GGG^{\alpha_1}$ and $\GGG^{\alpha_2}$ are different graph classes. It follows from Theorem \ref{thm-0123minim} each one defines a different minimal hereditary class of graphs of unbounded clique-width.
\end{proof}

%
%
%
%
%
%
\section{Recurrent but not almost periodic words} \label{recur}

We have seen that (with the exception of the all-zeros word) every almost periodic word $\alpha$ over $\{0,1,2,3\}$ defines a minimal hereditary class $\GGG^\alpha$ of unbounded clique width. At the other extreme, if $\alpha$ is a word over $\{0,1,2,3\}$ that contains a factor $\beta=\alpha_j \alpha_{j+1} \cdots \alpha_{j+k-2}$ that either does not repeat, or repeats only a finite number of times, then $\GGG^\alpha$ cannot be a minimal class of unbounded clique-width, as forbidding the induced subgraph $H^\alpha_{1,j}(k,k)$ would leave a proper subclass that by Theorem~\ref{thm-0123unbound} still has unbounded clique-width. 

Thus, to complete the delineation between minimality and non-minimality (with respect to having unbounded clique-width) of the classes $\GGG^\alpha$, it remains to consider words $\alpha$ that are recurrent but not almost periodic, i.e.\ words in which each factor occurs infinitely many times, but where the gap between consecutive occurrences of a factor may be arbitrarily large.


Fix a recurrent but not almost periodic word $\alpha$ over $\{0,1,2,3\}$. Since $\alpha$ is recurrent, any factor $\beta$ of $\alpha$ must occur an infinite number of times, and we will call the factors between any consecutive pair of occurrences of $\beta$ the \emph{$\beta$-gap factors}. Since $\alpha$ is not almost periodic, there exists a factor $\beta$ of length $k-1$, say, such that the $\beta$-gap factors can be arbitrarily long. Denote the sequence of $\beta$-gap factors by $\gamma_1,\gamma_2,\dots$. If, amongst these gap factors, we find that for any integer $m$ there exists some (indeed, infinitely many) $\gamma_i$ which has at least $m$ letters that are not 0, then by the analysis in Section~\ref{Sect:Unbounded}
there exist graphs whose clique-width grows as a function of $m$. Thus, the proper subclass $\CCC=\Free(H^\alpha_{1,j}(k,k))\cap \GGG^\alpha$ (where $j$ denotes the start of the first occurrence of $\beta$ in $\alpha$) contains graphs of arbitrarily large clique-width, and thus $\GGG^\alpha$ is not minimal.

Now let $\Gamma$ denote the collection of all recurrent words $\alpha$ over $\{0,1,2,3\}$ other than the all-zeros word, with the property that for any factor $\beta$ of $\alpha$, the weight of every $\beta$-gap factor is bounded. We now show that it is precisely the words in $\Gamma$ that define minimal classes of unbounded clique-width.

\begin{thm}\label{theorem-recur}
Let $\gamma$ be an infinite sequence over $\{0,1,2,3\}$ other than the all-zero sequence. Then $\GGG^\gamma$ is a minimal hereditary graph class of unbounded clique-width if and only if $\gamma\in\Gamma$.
\end{thm}

\begin{proof}
If $\GGG^\gamma$ is a minimal hereditary graph class of unbounded clique-width, and $\gamma$ is not almost periodic, then from the preamble to Section~\ref{recur} we have already demonstrated that $\gamma \in \Gamma$.

To prove the converse, suppose $\gamma \in \Gamma$. In the case that $\gamma$ is almost periodic, we may appeal directly to Lemma~\ref{lemma-minim}. For this more general setting, we may proceed in an almost identical manner.

If $\CCC$ is a proper hereditary subclass of $\GGG^\gamma$ then there must exist a non-trivial finite forbidden graph $F$ that is in $\GGG^\gamma$ but not in $\CCC$. In turn, this graph $F$ must be an induced subgraph of some $H^\gamma_{1,j}(k,k)$ for some $k\in \mathbb{N}$. Any graph $G$ in $\CCC$ must be $\Free(H^\alpha_{1,j}(k,k))$ for the fixed value of $k \ge 2$. 

As before, let $\beta=\gamma_j \gamma_{j+1} \cdots \gamma_{j+k-2}$ and $G^*$ denote the subgraph of $G$ induced by the columns $C_j \dots C_{j+k-1}$ . We can use the same cluster graph arguments to show that there is a partition $X_{V_G}/Y_{V_G}$ of the vertices of $G^*$ such that $\mu_{G^*}(X_{V_G}) \le 4k^2-3k$. Symmetrically, $\mu_{G^*}(Y_{V_G}) \le 4k^2-3$. 

We know that the factor $\beta$ appears an infinite number of times in $\gamma$ and that the weight of the string between each copy of $\beta$ is bounded by a constant, say, $W(\beta)$.   

Suppose the $i$-th copy of $\beta$ in $\gamma$ generates the subgraph $G^*_i$ of $G$, with corresponding partition $X_i/Y_i$, then we define $U_i$ as the subgraph induced by the vertices of $Y_{i-1}$, $X_{i}$ and all the vertices of $G$ in columns between these two sets.

This gives us a partition of $G$ such that for every $i$:
\begin{enumerate}[label=(\alph*)]
\item by using Corollary \ref{cor-H-cwd} the clique-width of $G(U_i)$ is at most $6(2k+W(\beta))+3$, and 
\item $\mu(U_i) \le (8k^2-6k)$ and $\mu(U_1 \cup \cdots \cup U_i) \le (8k^2-6k)$,
\end{enumerate}
So from Lemma \ref{lemma-partitions} the clique-width of $G$ is at most $(6(2k+W(\beta))+3) \times (8k^2-6k)$.

But we know that $k$ and $W(\beta)$ are fixed dependent on the forbidden graph $F$ and hence the graph class $\CCC$ has bounded clique-width.
Thus $\GGG^\gamma$ is a minimal hereditary graph class of unbounded clique-width.
\end{proof}

While $\Gamma$ includes every periodic and almost periodic word over $\{0,1,2,3\}$, it does also contain other (recurrent) words. One simple way to generate such sequences is by substitution. We use~\cite{fogg:substitutions:} as our reference work on substitutions. For example, consider the infinite binary word $\psi$ generated by an iterative substitution $\sigma$, beginning with $1$ such that  $\sigma(1)=1010$ and $\sigma(0)=0$. 

If we denote $\sigma^n(1)$ as the $n$-th iteration beginning with $\sigma^0(1)=1$  then
\[
\sigma^n(1)=\sigma^{n-1}(\sigma(1))=\sigma^{n-1}(1)\ 0\ \sigma^{n-1}(1)\  0.
\]
The first four iterates, and the start of $\psi$, are as follows.
\begin{align*}
\sigma^1(1)	&= 1010 \\
\sigma^2(1)	&= 1010\,0\,1010\,0 \\
\sigma^3(1)	&= 1010010100\,0\,1010010100\,0 \\
\sigma^4(1)	&= 1010010100010100101000\,0\,
		 			1010010100010100101000\,0 \\
\psi	&=	10100101000101001010000101001010001010010100000
10100101000\dots
\end{align*}	
The word $\psi$ has the following characteristics. 
\begin{enumerate}[label=(\roman*)]
\item The number of ones doubles with each iteration and therefore $\psi$ contains an infinite number of ones.
\item $\psi$ is a fixed point of $\sigma$ (i.e. $\sigma(\psi)=\psi$).
\item By construction $\psi$ is recurrent but is not almost periodic, because it contains arbitrarily long strings of zeros. 
\end{enumerate}

The following lemma shows that $\psi\in\Gamma$, and therefore provides us with the promised counterexample to the conjecture of Collins et al~\cite{collins:infinitely-many:}.
\begin{lemma}\label{weightbound}
For any factor $\beta$ of the word $\psi$, the weight of the $\beta$-gap factors is bounded, and thus $\psi\in\Gamma$. 
\end{lemma}

\begin{proof}
Suppose the longest subfactor of contiguous zeros in $\beta$ is $0^k$. It can be observed that $\sigma^{n}(1)$ ends with the factor $0^n$.  Hence $\beta$ must have appeared by the $(k+1)$-th iteration, $\sigma^{k+1}(1)$ or it is not a factor of $\psi$. Since $|\sigma^{k+1}(1)|_1 =2^{k+1}$, we have this as a bound on the weight between any consecutive occurrences. 
\end{proof}

We can extend this idea to construct other recurrent but not almost periodic infinite binary sequences in $\Gamma$. Indeed, any iterative substitution $\sigma_{\gamma}$ where $\sigma_{\gamma}(1)=\delta$ and $\sigma_{\gamma}(0)=0$ such that $\delta$ is a finite binary word whose first letter is 1, last letter is 0, and with $|\delta|_1 \ge 2$ will define a sequence $\gamma$. Now  $|\sigma^n_{\gamma}(1)|_1=|\delta|_1^n$ and it follows using Lemma~\ref{weightbound} that the weight of the set of $\beta$-gap factors for every factor $\beta$  is bounded, and hence $\gamma\in\Gamma$.

Finally, notice that $\Gamma$ does not comprise all recurrent binary sequences. Indeed, for any $\gamma\in\Gamma$ that is recurrent but not almost periodic, then the sequence $\overline{\gamma}$, formed as the complement of $\gamma$ (i.e.\ inverting the $1$s and $0$s), is a recurrent sequence that does not lie in $\Gamma$, and so $\GGG^{\overline{\gamma}}$ is not a minimal hereditary graph class of unbounded clique-width.

%
%
%
%
%
%
\section{Concluding remarks}

\paragraph{Linear clique-width} The \emph{linear clique-width} of a graph $G$ is defined as the minimum number of labels needed to construct $G$ by means of the operations allowed for standard clique-width, except for the disjoint union operation. Our minimality of unbounded clique-width arguments rest on constructing partitions that satisfy the conditions in Lemma~\ref{lemma-partitions}. In fact, there exists a `linear' analogue of this, see~\cite[Lemma 3]{collins:infinitely-many:}, and it is likely that this may be used in conjunction with our arguments above to show that $G^\alpha$ for any $\alpha\in\Gamma$ is also minimal of unbounded linear clique-width.

\paragraph{Towards a characterisation of clique width for bipartite graphs} 
While the ultimate goal of characterising which hereditary graph classes have unbounded clique-width remains somewhat remote, a nearer goal is the restriction of this characterisation to cover classes of \emph{bipartite} graphs. 

The identification of the collection of words $\Gamma$ represents a key step towards a fuller classification: even though we now have uncountably many minimal classes, the collection $\Gamma$ is relatively easily stated, and gives us the precise delineation between minimal and non-minimal for the classes under consideration.

To extend our work to cover all bipartite graphs still faces a number of hurdles. First, there exist minimal classes of bipartite graphs that are not of the form $\GGG^\alpha$ for any $\alpha\in\Gamma$ (for example, the bichain graphs of Atminas, Brignall, Lozin and Stacho~\cite{abls:minimal-classes-of:}), so the current four-letter alphabet $\{0,1,2,3\}$ is certainly not complete. Second, even with a more complete construction of classes, one must prove that such a list is complete, taking into account the pernicious issue of the class of square grids (which is bipartite and has unbounded clique-width yet contains no minimal class). 

\paragraph{Acknowledgements} We are grateful to Reem Yassawi for helpful discussions concerning recurrent and almost periodic sequences.

\bibliographystyle{plain}
\bibliography{refs}

\end{document}